\documentclass[11pt]{article}
\usepackage[a4paper]{geometry}
\usepackage{amsfonts}
\usepackage{amsmath}
\usepackage{amssymb}
\usepackage{amsthm}
\usepackage{fourier}
\usepackage{graphicx,color}
\usepackage{physics}
\usepackage{enumitem}
\usepackage{pifont}
\usepackage[font=small,labelfont=bf]{caption}
\usepackage{titlepic}
\usepackage{hyperref}
\hypersetup{linktoc=page}
\usepackage[capitalize, noabbrev]{cleveref}

\setitemize[0]{label=---}

\DeclareMathOperator{\Span}{Span}
\DeclareMathOperator{\Lie}{Lie}

\DeclareMathOperator{\ad}{ad}

\DeclareMathOperator{\Diff}{Diff}
\DeclareMathOperator{\Ve}{Vec}
\DeclareMathOperator{\DHa}{DHam}
\DeclareMathOperator{\PWC}{PWC}
\DeclareMathOperator{\Vol}{Vol}

\newcommand{\p}{\partial}
\newcommand{\f}{\overrightarrow}
\newcommand{\e}{\varepsilon}

\newcommand{\om}{\omega}
\newcommand{\C}{{\mathbb C}}
\newcommand{\R}{{\mathbb R}}
\newcommand{\N}{{\mathbb N}}
\newcommand{\Z}{{\mathbb Z}}
\newcommand{\T}{{\mathbb T}}
\newcommand{\LL}{{\mathcal L}}

\newcommand{\1}{{\mathbb 1}}

\newcommand{\ro}{{\rho}}
\newcommand{\n}{{\eta}}

\newcommand{\smooth}{{\mathcal{C}^{\infty}}}

\newcommand{\ov}{\overline}

\newtheorem{theorem}{Theorem}
\newtheorem{lemma}[theorem]{Lemma}
\newtheorem{prop}[theorem]{Proposition}

\newtheorem{cor}[theorem]{Corollary}
\newtheorem{definition}[theorem]{Definition}
\newtheorem{rem}[theorem]{Remark}

\title{Orbits and attainable Hamiltonian diffeomorphisms of mechanical Liouville equations}
\author{Bettina Kazandjian\thanks{Sorbonne Université, Université Paris Cité, CNRS, Inria, Laboratoire Jacques Louis-Lions, Paris, France (bettina.kazandjian@sorbonne-universite.fr, mario.sigalotti@inria.fr)},\quad Eugenio Pozzoli\thanks{Univ Rennes, CNRS, IRMAR - UMR 6625, F-35000 Rennes, France (eugenio.pozzoli@univ-rennes.fr)},\quad Mario Sigalotti}

\begin{document}

\maketitle

\begin{abstract}
We study the approximate controllability problem for Liouville transport equations along a mechanical Hamiltonian vector field. Such PDEs evolve inside the orbit $$\mathcal{O}(\rho_0):=\left\{\rho_0\circ \Phi\mid \Phi\in {\rm DHam}(T^*M)\right\},\quad \rho_0\in L^r(T^*M,\R), \quad r\in[1,\infty),$$
where $\rho_0$ is the initial density and ${\rm DHam}(T^*M)$ is the group of Hamiltonian diffeomorphisms of the cotangent bundle manifold $T^*M$. The approximately reachable densities from $\rho_0$ are thus contained in $\ov{\mathcal{O}(\rho_0)}$, where the closure is taken with respect to the $L^r$-topology. Our first result is a characterization of $\ov{\mathcal{O}(\rho_0)}$ when the manifold $M$ is the Euclidean space $\R^d$ or the torus $\T^d$ of arbitrary dimension: $\ov{\mathcal{O}(\rho_0)}$ is the set of all the densities whose sub- and super-level sets have the same measure as those of $\rho_0$. This result is an approximate version, in the case of ${\rm DHam}(T^*M)$, of a theorem by J.~Moser (Trans. Am. Math. Soc. 120: 286-294, 1965) on the group of diffeomorphisms.

We then present two examples of systems, respectively on $M=\R^d$ and $\T^d$, where the small-time approximately attainable diffeomorphisms coincide with ${\rm DHam}(T^*M)$, respectively at the level of the group and at the level of the densities.

The proofs are based on the construction of Hamiltonian diffeomorphisms that approximate suitable permutations of finite grids, and Poisson bracket techniques.
\end{abstract}



\tableofcontents

\section{Introduction}
\subsection{The model}

Let $M$ be a smooth $d$-dimensional boundaryless Riemannian manifold.
Its cotangent bundle $T^{*}M$ is equipped with canonical symplectic and volume forms, that in local position and momentum coordinates $(q,p)=(q_1,\dots,q_d,p_1,\dots,p_d)$ can be expressed, respectively, as $\om = \sum_{i=1}^d dp_i \wedge dq_i$ and $\om\wedge\dots\wedge \om$ ($d$-times).  For any measurable set $A\subset T^* M$ we denote by $\Vol(A)\geq 0$ its volume.

In this article we study the controllability of Liouville equations describing the transportation of a density along a Hamiltonian vector field. Let $\smooth(T^* M,\R)$, ${\rm Ham}(T^* M)$, ${\rm DHam}(T^* M),$ be, respectively, the space of smooth ($\smooth$) functions, Hamiltonian vector fields, and Hamiltonian diffeomorphisms. The space of real-valued functions $\smooth(T^* M,\R)$ is a Lie algebra equipped with the Poisson bracket $\left\{f,g\right\}=\sum_{j=1}^d \partial_{p_j}f\partial_{q_j}g-\partial_{q_j}g\partial_{p_j}f.$
With any smooth function $f$ one can associate the corresponding Hamiltonian vector field $\f{f}=\left\{f,\cdot\right\}$. In particular, in this work we shall restrict ourselves to mechanical Hamiltonians, i.e. , Hamiltonians of the form
\begin{equation}\label{eq:mechanical-hamiltonian}
H_{u(t)}(q,p)=\frac{|p|^2}{2}+V_0(q)+\sum_{j=1}^m u_j(t)V_j(q),\quad (q,p)\in T^*M,
\end{equation}
which are the sum of a drift term, denoted as $H_0(q,p)=\frac{|p|^2}{2}+V_0(q)$ (the sum of kinetic and potential energy), and a potential $\sum_{j=1}^m u_j(t)V_j(q)$ controlled through  $u(\cdot)=(u_1(\cdot),\dots,u_m(\cdot))\in \PWC(\R_+,\R^m)$, were $\PWC$ is used to denote piecewise constant functions. We suppose in the following that for each $u\in \R^m$ the vector field $\f{H}_u=\left\{H_u,\cdot\right\}$ is complete.

Let us fix $r\in [1,\infty)$.
 Given $u(\cdot)\in\PWC(\R_+,\R^m)$, 
we consider the transportation of an initial density $\rho_0\in L^r(T^* M,\R)$ along the piecewise constant Hamiltonian vector field $t\mapsto \f{H}_{u(t)} \in {\rm Ham}(T^* M)$, given by the Liouville equation
\begin{equation}\label{eq:liouville}
\begin{cases}
\partial_t\rho(q,p,t)=\overrightarrow{H}_{u(t)}(q,p)\rho(q,p,t),&\\
\rho(q,p,0)=\rho_0(q,p).&
\end{cases}
\end{equation}
Equation \eqref{eq:liouville} more explicitly reads as the partial differential equation
$$\partial_t\rho(q,p,t)=p \cdot \nabla_q \rho(q,p,t)-\nabla_qV_0(q)\cdot \nabla_p\rho(q,p,t)-\sum_{j=1}^m u_j(t)\nabla_qV_j(q)\cdot \nabla_p\rho(q,p,t).$$
This is a bilinear (hence, non-linear) control problem, with infinite-dimensional state space $L^r(T^*M,\R)$.

Since the vector field $\f{H}_u$ is complete for every $u\in \R^m$, the 
Liouville equation is globally-in-time well posed, meaning that for any $\rho_0\in L^r(T^*M,\R),u(\cdot)\in \PWC(\R_+,\R^{m})$ there exists a unique mild solution $(t\mapsto \rho(t)=\rho(t;u(\cdot),\rho_0))\in \mathcal{C}^0(\R,L^r(T^*M,\R))$ of \eqref{eq:liouville}.
In this article we are interested in studying the small-time reachable sets of \eqref{eq:liouville}. 
\begin{definition}[Small-time reachable densities]
A density $\rho_1\in L^r(T^* M)$ is said to be \emph{small-time reachable for \eqref{eq:liouville}} if for every time $T> 0$ there exist a smaller time $\tau\in [0,T]$ and a control $u(\cdot)\in \PWC([0,\tau],\R^m)$ such that  $\rho(\tau;u(\cdot),\rho_0)=\rho_1$.
\end{definition}

 By the method of characteristics, the solution of \eqref{eq:liouville} can be written as
\begin{equation}\label{eq:solution}
\rho(t)=\rho_0\circ \Phi^t_{H_u},
\end{equation}
where $\Phi^t_{H_u}\in {\rm DHam}(T^* M)$ denotes the flow  on $T^* M$ of the non-autonomous Hamilton equation, which in local coordinates reads as
\begin{equation} \label{eq:hamilton}
    \left\{ \begin{array}{ll}
         \dot{q}=\nabla_pH_{u(t)}(q,p), \\
         \dot{p}=-\nabla_{q}H_{u(t)}(q,p). 
    \end{array}
    \right.
\end{equation}
The study of the reachable densities for \eqref{eq:liouville} can then be naturally lifted to the study of the reachable Hamiltonian diffeomorphisms for \eqref{eq:hamilton}. In fact, as one expects, controllability in the group of Hamiltonian diffeomorphisms implies controllability of the associated Liouville equation. 
\subsection{Approximate orbits of Hamiltonian diffeomorphisms on densities}
Due to \eqref{eq:solution}, smooth initial densities are preserved along the flow of \eqref{eq:liouville}, hence exact controllability in $L^r$-spaces is clearly impossible. We shall thus focus on the $L^r$-approximate controllability problem.
\begin{definition}[Small-time approximately reachable densities]
A density $\rho_1\in L^r(T^* M)$ is said to be 
\emph{small-time approximately reachable for \eqref{eq:liouville}} if for any $\varepsilon>0$ there exist a time $\tau\in[0,\varepsilon]$ and a control $u(\cdot)\in \PWC([0,\tau],\R^m)$ such that $\|\rho(\tau;u(\cdot),\rho_0)-\rho_1\|_{L^r(T^* M)}<\varepsilon$. In this case, we write $\rho_1\in \overline{R_{\rm st}}(\rho_0)$.
\end{definition}

According to \eqref{eq:solution}, the solutions of \eqref{eq:liouville} evolve in the set
\begin{equation}\label{eq:orbit}
\mathcal{O}(\rho_0)=\left\{\rho_0\circ \Phi\mid \Phi\in {\rm DHam}(T^* M)\right\}, 
\end{equation}
that is, the orbit of the group action of ${\rm DHam}(T^* M)$ on $L^r(T^* M)$ passing through $\rho_0$. As a consequence, the solution of system \eqref{eq:liouville} can approximately reach at most any density belonging to the $L^r$-closure of $\mathcal{O}(\rho_0)$: that is, 
$$\overline{R_{\rm st}}(\rho_0)\subset  \overline{\mathcal{O}(\rho_0)}.$$
Our first result is a characterization of the orbits' closure, when the underlying manifold is the torus or the Euclidean space of arbitrary dimension.
\begin{theorem}[$L^r$-closure of the orbits]\label{thm:closure-orbit}
Let $M$ be $\T^d=\R^d/2\pi\Z^d$ or $\R^d$ and $r\in[1,\infty)$. Then, for any $\rho_0\in L^r(T^*M)$,
\begin{equation}\label{eq:sub-levels}
 \overline{\mathcal{O}(\rho_0)}=\left\{\rho_1\in L^r(T^*M)\mid \forall \mu<\nu, \Vol(\{\mu<\rho_1<\nu\})=\Vol(\{\mu<\rho_0<\nu\})\right\},
 \end{equation}
where the closure is taken in the $L^r$-topology.
\end{theorem}
 Theorem~\ref{thm:closure-orbit} can be seen as the approximate version for Hamiltonian diffeomorphisms of a theorem by J.~Moser on the group of diffeomorphism ${\rm Diff}(M)$ of a connected boundaryless compact manifold $M$ \cite{moser}. More precisely, for any $\rho_0\in \smooth(M,(0,\infty))$ one can define the orbit of a group action of ${\rm Diff}(M)$ on $\smooth(M,(0,\infty))$ passing through $\rho_0$ as follows
$$\mathcal{O}_{\rm Diff}(\rho_0)=\left\{\det(DP)(\rho_0\circ P)\mid P\in {\rm Diff}(M)\right\}.$$
Then, Moser's theorem states that 
$$\mathcal{O}_{\rm Diff}(\rho_0)=\left\{\rho_1\in \smooth(M,(0,\infty))\mid \int_M \rho_1(x)dx=\int_M\rho_0(x)dx\right\}.$$
An exact counterpart for Hamiltonian diffeomorphisms of Moser's theorem should also take into account the topology of the sub- and super-level sets of $\rho_0$, which plays no role in the approximate version stated in Theorem~\ref{thm:closure-orbit}. This could be the subject of future research.
\\ \\
We now define the notion of approximate controllability for \eqref{eq:liouville}.
\begin{definition}[Small-time approximate controllability]
Equation \eqref{eq:liouville} is said to be \emph{small-time approximately controllable in $L^r(T^*M)$} if, for every $\rho_0\in L^r(T^*M)$ , 
$$\overline{R_{\rm st}}(\rho_0)= \overline{\mathcal{O}(\rho_0)}.$$
\end{definition}
In some situations it is easier to look at the approximate controllability problem lifted on the group ${\rm DHam}(T^* M)$. Since the simplectic 2-form on $T^*M$ is the differential of the Liouville 1-form $\sum_{i=1}^dp_idq_i$, it is exact. As a consequence, the flux group of $T^*M$ is trivial \cite[Proposition 3.47]{viterbo-notes}, implying that ${\rm DHam}(T^*M)$ is $\mathcal{C}^1$-closed \cite{ono}. This motivates the following notion of approximate reachability and controllability at the level of the group ${\rm DHam(T^*M)}$. 

\begin{definition}[Small-time approximately reachable diffeomorphisms and controllability] A Hamiltonian diffeomorphism $\Psi\in {\rm DHam}(T^* M)$ is said to be \emph{small-time approximately reachable for \eqref{eq:hamilton}} if for every $\varepsilon>0$, every compact $K\subset T^* M$, and every $\ell\in \N$ there exist a time $\tau\in[0,\varepsilon]$ and a control $u(\cdot)\in \PWC([0,\tau],\R^m)$ such that $\|\Phi^\tau_{H_u}-\Psi\|_{\mathcal{C}^\ell(K)}<\varepsilon$. In this case, we write $\Psi\in \overline{\mathcal{R}_{\rm st}}$.
\\ 
Equation \eqref{eq:hamilton} is said to be 
\emph{small-time approximately controllable in ${\rm DHam}(T^* M)$} if 
$$ \overline{\mathcal{R}_{\rm st}}={\rm DHam}(T^* M).$$
\end{definition}
This is a non-linear control problem, with infinite-dimensional state space ${\rm DHam}(T^*M)$.

We remark that the previous notion is equivalent to the notion of approximate controllability in the compact-open topology (whose definition is recalled in Section~\ref{sec:propertiesVF}).
We also notice that the definition can be extended to the case where the vector fields $\f{H_u}$ are not necessarily complete, up to considering flows on compacts.

As previously announced, approximate controllability at the level of the group implies approximate controllability at the level of the densities.
\begin{lemma} \label{lem5}
    If $\Psi \in \DHa(T^* M)$ is such that $\Psi\in \overline{\mathcal{R}_{\rm st}}$, then $\ro_0 \circ \Psi \in \overline{R_{\rm st}}(\rho_0)$ for every $\ro_0 \in L^r(T^* M)$.
\end{lemma}

\subsection{Examples of approximate controllability in ${\rm DHam}(T^*\R^d)$ and $L^r(T^*\T^d,\R)$}
The scope of the second part of the paper is to present two examples of systems 
for which 
the previously introduced controllability properties hold. 
The first system that we consider is posed in an Euclidean space of arbitrary dimension, the second one in a torus of arbitrary dimension (both equipped with the standard simplectic form).

\subsubsection{A system in $\R^d$}

\noindent
We consider the mechanical Hamiltonian
\begin{equation}\label{eq:euclidean}
H_u(q,p)=\frac{|p|^2}{2}+V_0(q)+\sum_{j=1}^d u_j q_j+u_{d+1}e^{-|q|^2/2},\qquad (q,p)\in T^*\R^d=\R^d_q\times \R^d_p,
\end{equation}
with $V_0\in \smooth(\R^d_q,\R)$ such that $\nabla
V_0$ is globally Lipschitz continuous (this last hypothesis is used to guarantee that each Hamiltonian vector field $\f{H_u}$ is complete). 
Here and in the following $\smooth(\R^d_q,\R)$ is used to denote the space of functions on $T^*\R^d$ that depend only on the state variable $q$. Similarly, we will consider later $\smooth(\R^d_p,\R)$ and $\smooth(\T^d_q,\R)$.
Since the vector fields $\f{q_1}=-\partial_{p_1},\dots,\f{q_d}=-\partial_{p_d}$, $\f{e^{-|q|^2/2}}=e^{-|q|^2/2}\sum_{j=1}^d q_j \partial_{p_j}$ are also globally Lipschitz continuous, system \eqref{eq:hamilton}, \eqref{eq:euclidean} (and hence also \eqref{eq:liouville}, \eqref{eq:euclidean}) is globally-in-time well posed. 
Here and in similar situations in the rest of the paper, with a slight abuse of notation,
$\f{q_i}$ stays for $\f{\varphi}$ where $\varphi:(q,p)\mapsto q_i$ and similarly for $\f{e^{-|q|^2/2}}$. 

For the system associated with the Hamiltonian introduced in \eqref{eq:euclidean} we prove the following result.
\begin{theorem}\label{thm:euclidean}
The Hamilton equations \eqref{eq:hamilton}, \eqref{eq:euclidean} are small-time approximately controllable in ${\rm DHam}(T^*\R^d)$.
\end{theorem}
As a consequence of  Lemma~\ref{lem5} and Theorem~\ref{thm:euclidean}, we also get the following result.
\begin{cor}
The Liouville equation \eqref{eq:liouville}, \eqref{eq:euclidean} is small-time approximately controllable in $L^r(T^*\R^d)$, $r\in[1,\infty)$.
\end{cor}

\subsubsection{A system in $\T^d$}

\noindent
We consider the mechanical Hamiltonian
\begin{equation}\label{eq:torus}
H_u=\frac{|p|^2}{2}+V_0(q)+\sum_{j=1}^d u_{2j-1}\cos(k_j\cdot q)+u_{2j}\sin(k_j\cdot q),\quad (q,p)\in T^*\T^d=\T^d_q\times \R^d_p,
\end{equation}
where $V_0\in \smooth(\T^d_q,\R)$ and 
\begin{equation}\label{eq:frequencies}
k_1=(1,0,\dots,0),\dots,k_{d-1}=(0,\dots,0,1,0),\quad k_d=(1,\dots,1).
\end{equation}
For such system we prove the following result.
\begin{theorem}\label{thm:torus}
The Liouville equation \eqref{eq:liouville}, \eqref{eq:torus} is small-time approximately controllable in $L^r(T^*\T^d)$, $r\in[1,\infty)$.
\end{theorem}
To the best of our knowledge, it is an open problem whether the analogue of Theorem~\ref{thm:euclidean} for the Hamilton equations \eqref{eq:hamilton}, \eqref{eq:torus} holds. However, it is at least possible to control finite ensembles of points of arbitrary cardinality in $T^*\T^d$. Let us stress that such controllability property is not only approximate but also \emph{exact}, as shown in Section~\ref{sec:8}. 

\subsection{Proof strategy} 
Concerning Theorem \ref{thm:closure-orbit}, the inclusion $\ov{\mathcal{O}(\rho_0)}\subset\mathcal{L}(\rho_0)$ readily follows from the fact that Hamiltonian diffeomorphisms are measure-preserving, in the sense that $\det (D\Phi)\equiv 1$ (where $D\Phi$ denotes the Jacobian matrix of a Hamiltonian diffeomorphism $\Phi$). We show the converse inclusion $\mathcal{L}(\rho_0)\subset\ov{\mathcal{O}(\rho_0)}$ by explicitly constructing an approximating Hamiltonian diffeomorphism. This is done in two steps: first, we approximate a rearrangement of any $\rho_1$ in the right-hand side of \eqref{eq:sub-levels} with a permutation acting on a finite grid of cubes (see Lemma \ref{lem9}); then, we approximate the permutation of cubes with localized Hamiltonian diffeomorphisms (see Lemma \ref{lem10}). The first step of our construction reminds of a result by P.D.~Lax on the approximation of measure-preserving diffeomorphisms by permutations \cite{lax} (for similar constructions, see also the work by Y.~Brenier and W.~Gangbo \cite{brenier}). Note, however, that our result is conceptually reciprocal, meaning that we approximate permutations with Hamiltonian (hence, in particular, measure-preserving) diffeomorphisms.

\medskip

Our study of the approximately attainable Hamiltonian diffeomorphisms, and densities, is based on the identification of some finite families of mechanical Hamiltonians generating the group of Hamiltonian maps. It has been recently proved by Berger and Turaev \cite{berger} that any Hamiltonian diffeomorphism, on the torus or Euclidean space's cotangent bundles, is generated by two abelian subgroups of ${\rm DHam}(T^*M)$, namely horizontal and vertical shears. We thus focus on the approximation of the elements of such subgroups.

However, a major constraint is present: due to the presence of the uncontrolled drift term (the kinetic and potential energy $|p|^2/2+V_0(q)$) in \eqref{eq:mechanical-hamiltonian}, not all Poisson brackets are (a priori) attainable Hamiltonians, but only those corresponding to forward flows along the drift. We can circumvent this difficulty on ${\rm DHam}(T^*\R^d)$, roughly thanks to the existence in this setting of a Hamiltonian function quadratic in $p$ (namely, the harmonic oscillator Hamiltonian $(|p|^2+|q|^2)/2$) whose flow is periodic in the group ${\rm DHam}(T^*\R^d)$, since its period does not depend on the initial configuration $(q,p)$). This property allows to approximate also backward propagation along the drift kinetic energy, and hence all the needed Poisson brackets (for the details, see Proposition \ref{th3}).

On ${\rm DHam}(T^*\T^d)$, the harmonic oscillator Hamiltonian is not globally defined and hence we need to work locally. In particular, in this setting we are able to show that the orbits of the group action of ${\rm DHam}(T^*\T^d)$ on $L^r(T^*\T^d)$-densities are approximately attainable. This is a weaker result than the one showed on ${\rm DHam}(T^*\R^d)$.
An important point here is that, at the level of densities, we can approximate the action of some non-Hamiltonian diffeomorphisms, which are not allowed at the level of the group ${\rm DHam}(T^*\T^d)$ (see Lemma~\ref{lem6}). Such additional non-Hamiltonian movement allows to approximate also backwards propagation along the drift kinetic energy, but only at the level of the densities (hence not at the level of the group ${\rm DHam}(T^*\T^d)$).

\subsection{Some additional literature}
Attainable diffeomorphisms by dynamical systems is 
the subject of a rich literature in various domains. The problem of attainability for volume-preserving diffeomorphisms appears e.g. in fluidodynamics \cite{marsden2,shnirelman}, as well as in shape optimization \cite{trelat-2015}, where they furnish the natural configuration space. More in general, controlling diffeomorphisms is also useful in understanding which output can be approximated in deep learning architectures described by neural ODEs \cite{agrachev-sarychev3,Elamvazhuthi_et_al,zuazua}, since diffeomorphisms act transitivily on any finite set of training data.

This paper fits in a series of previous work on the controllability of diffeomorphisms group, initiated in \cite{agrachev-caponigro} and \cite{trelat-2017} on local exact controllability, and continued in the recent works \cite{agrachev-sarychev2,agrachev-sarychev3} on global approximate controllability, which correspond to the properties studied in Theorems \ref{thm:euclidean} and \ref{thm:torus}.

This work also represents the classical counterpart of the quantum study recently developed in \cite{beauchard-pozzoli2}, concerning attainable diffeomorphisms and Schrödinger equations. The results are similar, but here we are able to study both controllability at the level of the group ${\rm DHam}(T^*M)$ and at the level of the densities in $L^r(T^*M,\R)$, whereas in the quantum study only the controllability at the level of the wave functions in $L^2(M,\C)$ was performed. 

The control problem for Liouville transport equations was previously considered also in \cite{brockett-liouville}, but in a different setting (non-Hamiltonian, and with space-dependent controls).

To the best of our knowledge, this is the first work establishing the approximate controllability of the group of Hamiltonian diffeomorphisms and Liouville equations for mechanical systems (that is, for systems of the form \eqref{eq:hamilton},\eqref{eq:mechanical-hamiltonian} or \eqref{eq:liouville},\eqref{eq:mechanical-hamiltonian}). Furthermore, the results are proved to hold in arbitrarily small times (notice that, due to the presence of a drift, small-time approximate controllability is not equivalent to approximate controllability).

\medskip

Section \ref{sec:2} is dedicated to the proof of Theorem \ref{thm:closure-orbit}. In Section \ref{sec:propertiesVF} we recall some tools on vector fields. In Section \ref{sec:4} we collect some properties of the attainable sets. In Section \ref{sec:5} we introduce two abelian subgroups that generate ${\rm DHam}(T^*M)$ for $M=\R^d,\T^d$. Sections \ref{sec:6} and \ref{sec:7} are dedicated, respectively, to the proofs of Theorems \ref{thm:euclidean} and \ref{thm:torus}. In the last Section \ref{sec:8} we show the small-time \emph{exact} controllability of finite ensembles of points for systems \eqref{eq:euclidean} and \eqref{eq:torus}.

\section{Proof of Theorem~\ref{thm:closure-orbit}}\label{sec:2}

In this section $M$ denotes $\R^d$ or $\T^d$. Given $\rho_0,\rho_1\in L^r(T^*M)$, we shall write $\rho_0\sim \rho_1$ if there exists $\Phi\in {\rm DHam}(V)$ such that $\rho_0=\rho_1\circ \Phi$ (it is an equivalence relation). Let 
$$\mathcal{L}(\rho_0):=\left\{\rho_1\in L^r(T^*M)\mid \forall \mu<\nu, \Vol(\{\mu<\rho_1<\nu\})=\Vol(\{\mu<\rho_0<\nu\})\right\}, $$
and we have to prove that 
\begin{equation}\label{eq:closure-orbit}
\ov{\mathcal{O}(\rho_0)}=\mathcal{L}(\rho_0),\qquad \forall \rho_0\in L^r(T^*M). 
\end{equation}

\subsection{Proof of the inclusion $\ov{\mathcal{O}(\rho_0)}\subset\mathcal{L}(\rho_0)$}
Given $\ro_1 \in \ov{\mathcal{O}(\ro_0)}$, there exists a sequence $(\phi_n)_{n\in \N}$ in $\DHa(V)$ such that $\ro_0 \circ \phi_n \rightarrow \ro_1$ in $L^r$ as $n\to\infty$. Let $\mu < \nu$ be in $\R$. Up to extracting a sub-sequence, $\ro_0 \circ \phi_n \rightarrow \ro_1$ almost everywhere. In particular, for a.e. $x \in \{y \mid \mu < \ro_1(y) < \nu\}$, $\ro_0 \circ \phi_n(x) \rightarrow \ro_1(x)$ as $n\to\infty$. So there exists a zero-measure set $\mathcal{N}\subset V$ such that 
\begin{equation*}
    \{\mu < \ro_1 <\nu\} \subset \cup_{N\in \N}\cap_{n \geq N}\{\mu < \ro_0\circ \phi_n < \nu\}\cup \mathcal{N}.
\end{equation*}
Then, 
\begin{align*}
    \Vol(\{\mu < \ro_1 < \nu\})&\leq \Vol(\cup_{N\in \N}\cap_{n \geq N}\: \{\mu < \ro_0\circ \phi_n < \nu\}) + \Vol(\mathcal{N})\\
    &= 
    \lim_{N \rightarrow\infty}\Vol(\cap_{n \geq N}\: \{\mu < \ro_0\circ \phi_n < \nu\}) \\
    &
    \leq \limsup_{N \rightarrow\infty}\Vol(\{\mu < \ro_0\circ \phi_N < \nu\})\\
    &= \Vol(\{\mu < \ro_0 < \nu\}).
\end{align*}
Conversely, since a Hamiltonian change of variables preserves the $L^r$-norm, we also have that $\ro_1 \circ \phi_n^{-1}\rightarrow \ro_0$ in $L^r$ as $n\to\infty$. Hence, 
the previous reasoning also shows that $\Vol(\{\mu<\ro_0< \nu\})\leq\Vol(\{\mu < \ro_1 < \nu\})$. 
This implies that $\Vol(\{\mu <\ro_0< \nu\})=\Vol(\{\mu < \ro_1 < \nu\})$.
\hfill$\square$

\begin{rem}
The proof of the inclusion $\ov{\mathcal{O}(\rho_0)}\subset\mathcal{L}(\rho_0)$ given in this section works for any choice of the manifold $M$, and not only for $M=\R^d$ and $M=\T^d$.
\end{rem}

\subsection{Proof of the inclusion $\mathcal{L}(\rho_0)\subset\ov{\mathcal{O}(\rho_0)}$}
\subsubsection{Approximation with a permutation of the mesh}
\begin{lemma} \label{lem8}
    Let $\ro_0,\ro_1\in L^r(T^*M)$ be such that $\ro_0 \sim \ro_1$. For every $\mu < \nu$  in $\R$ with $\mu \ne0$, we have
    \begin{equation*}
    \Vol(\{\mu \leq \ro_0 < \nu\})=\Vol(\{\mu \leq \ro_1 < \nu\}).
\end{equation*}
\end{lemma}
\begin{proof}
It is sufficient to prove that $\Vol(\{\ro_0=\mu\})=\Vol(\{\ro_1=\mu\})$ for every $\mu \ne 0$. The densities $\rho_0,\rho_1$ are in $L^r$, 
so if $\mu \neq 0$, there exists $n_0 \in \N^*$ such that $\Vol(\{\mu-\frac{1}{n_0}<\ro_i<\mu +\frac{1}{n_0}\})< +\infty$, for $i\in \{0,1\}$, and then 
\begin{align*}
    \Vol(\{\ro_1=\mu\}) &= \Vol\left(\cap_{n\in \N^*}\{\mu - \frac{1}{n}< \ro_1 < \mu + \frac{1}{n}\}\right) \\
    &= \lim_{n \rightarrow\infty}\Vol\left(\left\{\mu- \frac{1}{n}< \ro_1 < \mu + \frac{1}{n}\right\}\right) \\
    &= \lim_{n \rightarrow\infty}\Vol\left(\left\{\mu- \frac{1}{n}< \ro_0 < \mu + \frac{1}{n}\right\}\right) \\
    &= \Vol(\{\ro_0=\mu\}). 
\end{align*}
\end{proof}
In the following we introduce a regular mesh of $T^*M$, i.e., we cover $T^*M$ by a set of distinct cubes of volume $h^{2d}, h>0$. In order to work with cubes we use the norm 
$$\|x\| = \max_{i \in \llbracket 1,2d \rrbracket} |x^i|, \qquad  x \in T^*M,$$
where the notation $\llbracket 1,2d \rrbracket$ stands for $\left\{1,\dots,2d\right\}$. 
For $x \in T^*M$ and $h>0$, the open and the closed cubes of center $x$ and radius $h$ are given by 
$$C(x,h)=\{y \in T^*M \mid \|x-y\| < h\},\qquad \ov{C}(x,h)=\{y \in T^*M \mid \|x-y\| \le h\}.$$
We denote by $M_h=\left\{m_n \mid n \in \N\right\}$ a regular grid of points in $V$ such that 
$$T^*M = \cup_{n \in \N}\ov{C}(m_n,h),$$  
while the open cubes $C(m_n,h)$, $n\in\N$, are pairwise disjoint. We call such a $M_h$ a \emph{mesh of size $h$}.

\begin{definition}
Given $h>0$
and a mesh $M_h$ of size $h$, a map $F:T^*M\to T^*M$ is said to be a \emph{permutation of $M_h$} if $F$ is bijective and $F$ translates every open cube of the mesh to another one, that is, 
for every $n \in \N$ there exists 
$\ell \in \N$ such that, for every $x \in C(m_n,h)$, $F(x)=x+m_{\ell}-m_n$. 
\end{definition}
\begin{lemma} \label{lem9}
    Let $\ro_0,\ro_1 \in L^r(T^*M)$ 
    be such that $\ro_0 \sim \ro_1$. For every $\e >0$, there exists $h_0 > 0$ such that for every $h \in (0,h_0)$ and every mesh $M_h$ of size $h$, 
    there exists a permutation $F_{\e,h}$ of $M_h$  such that 
    $$\|\ro_0 \circ F_{\e,h} - \ro_1\|_{L^r}< \e.$$ 
\end{lemma}
\begin{proof}
\underline{First step:} 
Let 
$0 < a < A$. For $i \in \{0,1\}$, by dominated convergence we have 
$$\|\ro_i \1_{\{a < |\ro_i|< A\}} - \ro_i\|_{L^r} \underset{a \rightarrow 0,A \rightarrow + \infty}{\longrightarrow}0.$$
Moreover we still have $$\ro_0\1_{\{a < |\ro_0|< A\}}\sim \ro_1 \1_{\{a < |\ro_1|< A\}},$$
so in the following we can consider without loss of generality that there exist $0< a < A$ and a set $W \subset T^*M$ of finite measure, such that
$a< |\ro_0(x)|,| \ro_1(x)|< A$
for $x \in W$ and $\ro_0(x)=\ro_1(x)=0$
for $x\in T^*M \setminus W$. 
\\ \\
\underline{Second step:} Let us approximate the densities $\rho_0$ and $\rho_1$ by step functions. 
Given $N\in \N$ we set $\xi_0=-A$ and $\xi_{k}=\xi_{k-1} + \frac{2A}{N}$ for every $k \in \llbracket 1,N \rrbracket$. For $i \in \{0,1\}$, let 
$$I_N^i=\sum_{k=1}^{N}\xi_k\1_{\left\{\xi_{k-1} \leq \ro_i < \xi_{k}\right\}}, $$
and thanks to 
the previous step, we deduce that $I_N^i \underset{N \rightarrow\infty}{\longrightarrow}\ro_i$ $\text{in }L^r$.
Notice that for $N$ large enough, by the previous step, either  $\xi_{k-1}\ne 0$ (and $\Vol(\{\xi_{k-1} \leq \ro_i < \xi_{k}\})<+\infty$) or $\rho_i\equiv 0$ on $\{\xi_{k-1} \leq \ro_i < \xi_{k}\}$. 
\\ \\
\underline{Third step:} Let us prove now that the sub-level sets of $\rho_0$ and $\rho_1$ can be approximately covered by the same number of cubes. 
Consider $i\in \{0,1\}$, $\xi< \xi'$ in $\R$, 
 set 
$$V^i:=\left\{x \in T^*M \mid \xi \leq \ro_i(x) < \xi'\right\},$$
and assume that $\Vol(V^i)<+\infty$. \\
Let $W^i\subset T^*M$ be open and such that $V^i\subset W^i$, $\Vol(W^i)<\Vol(V^i)+\e$, where $\e>0$ is arbitrary.  
Set, moreover, for $\n >0$, 
$$W_{\n}^i:=\left\{x  \in W^i\mid d(x,\p W^i) > \n\right\},$$
where $d(x,\p W^i)=\inf_{y \in \p W^i}\|x-y\|$. 
By dominated convergence, $\1_{W^i_{\n}}\rightarrow\1_{W^i}$ in $L^r$ as $\n$ goes to zero.
%
In particular, $0< \Vol(W^i)-\Vol(W^i_{\n})<\e$, for $\n$ small enough, and 
$$\Vol(W^i_\n \Delta V^i)= \Vol(W^i_\n \setminus V^i) + \Vol (V^i \setminus W^i_\n)\leq \Vol(W^i \setminus V^i)+ \Vol(W^i \setminus W^i_\n) <2\e,$$
where we used $\Delta$ to denote the symmetric difference.
 In particular, $|\Vol(W^i_\n) - \Vol(V^i)|\leq \Vol(W^i_\n \Delta V^i)<2\e$.
Given $h>0$ and a mesh $M_h=\left\{m_n\mid n\in \N\right\}$ of size $h$, 
consider the minimal set $J_i\subset \N$ 
such that $W^i_{2h}\subset \cup_{n\in J_i}\ov{C}(m_n,h)$. 
Notice that 
if
$n \in J_i$ 
then $\ov{C}(m_n,h)\subset W^i$. 
Hence, for $h$ small enough,  
\begin{equation*}
   \Vol(W^i)-\e < \Vol(W^i_{2h})\leq \Vol(\cup_{n\in J_i}C
    (m_n,h))\leq \Vol(W^i)
\end{equation*}
from which we deduce that
\begin{equation}\label{eq:squizz}
\Vol(V^i\Delta (\cup_{n\in J_i}C
    (m_n,h)))\leq \Vol(W^i \setminus V^i)+ \Vol(W^i \setminus\cup_{n\in J_i}C
    (m_n,h)))<2\e.
\end{equation}
Now, denote by $\hat{J}_i$ the set of indices $n\in J_i$ such that 
$\Vol(C(m_n,h)\cap V_i)>\frac{1}{2}\Vol(C(m_n,h))$, so that 
\begin{equation}\label{eq:squizz2}
\e>\Vol(W^i\setminus V^i)\ge \sum_{n\in J_i\setminus \hat J_i}\Vol(C(m_n,h)\setminus V^i)\ge \frac12\Vol\left(\cup_{n\in J_i\setminus \hat{J}_i}C(m_n,h)\right).
\end{equation}
According to \eqref{eq:squizz} and \eqref{eq:squizz2}, we then have
\[\Vol(V^i)+\e>\Vol(\cup_{n\in\hat J_i}C(m_n,h))>\Vol(\cup_{n\in J_i}C(m_n,h))-2\e>\Vol(V^i)-4\e.
\]
Hence $|\Vol(\cup_{n \in \hat J_i}{C}(m_n,h))-\Vol(V^i)|<4\e$. In particular,   since $\Vol(V^0)=\Vol(V^1)<+\infty$ and
denoting by $\sharp \hat J_i$ the cardinal of $\hat J_i$,
it follows that $h^{2d}|\sharp \hat J_0 - \sharp \hat J_1|<8\e$. \\
Let us prove that we can approximate $V^0$ and $V^1$ by the same number of cubes. 
Pick 
$\Tilde{J}_1\subset \hat J_1$, $\Tilde{J}_2\subset \hat J_2$ such that $\sharp \Tilde{J}_1=\sharp \Tilde{J}_2=\min(\sharp \hat J_1,\sharp \hat J_2)$. Then, for $i\in \{0,1\}$,
\[
|\Vol(V^i)-\Vol(\cup_{n \in  \tilde J_i}C(m_n,h))|
\le |\Vol(V^i)-\Vol(\cup_{n \in  \hat J_i}C(m_n,h))|+(2h)^{2d}|\sharp \hat J_0 - \sharp \hat J_1|
<12\e.\]
\\ \\
\underline{Fourth step:} We fix $\e>0$, $N\in \N$ such that $\|I^i_N-\rho_i\|_{L^r}<\e$ for $i\in\{0,1\}$, and we apply the previous step to the each of the sub-level sets 
\[V_k^i=\{\xi_{k-1} \leq \ro_i < \xi_{k}\},\qquad i\in\{0,1\},\ k \in \llbracket 1,N \rrbracket,\]
with $0\not\in [\xi_{k-1},\xi_k)$. We recall that by construction, $\Vol(V^i_k)<+\infty$ for $0\not\in [\xi_{k-1},\xi_k)$, so we are indeed allowed to apply the third step. \\
Then, given $\n=\n(\e)>0$ to be fixed later, there exists $h_0>0$ such that for every $h\in (0,h_0)$, every mesh $M_h$ of size $h$, 
and every $k \in \llbracket 1,N \rrbracket$ with $0\not\in [\xi_{k-1},\xi_k)$, there exist $J_k^0,J_k^1 \subset \N$  such that $\sharp J_k^0 = \sharp J_k^1$ and 
\[|\Vol(V_k^i)- \Vol(\cup_{n \in  J_k^i}C(m_n,h))|< \n, \qquad i \in \{0,1\}.\]
Moreover,
the cubes indexed by $J_k^i$
intersect $V_k^i$ for more than half their volume, so, for $i$ fixed, the sets $J_k^i$ 
are disjoint. We can construct a permutation $F_{\e,h}$ of the mesh $M_h$ that translates the cubes of $J_k^0$ to the cubes of $J_k^1$.
Then, since  $F_{\e,h}$ preserves the volume, 
\begin{align*}
    \int_{T^*M}|\1_{V^0_k}\circ F_{\e,h}-\1_{V^1_k}|
    &= \Vol(V^1_k \setminus F_{\e,h}(V^0_k))+ \Vol(F_{\e,h}(V^0_k)\setminus V^1_k) \\
    &= \Vol(V^1_k \setminus F_{\e,h}(V^0_k))+ \Vol(V^0_k\setminus F_{\e,h}^{-1}(V^1_k)) \\
    &\leq \Vol(V^1_k \setminus \cup_{n\in J_k^1}C(m_n,h)) + \Vol(V^0_k \setminus \cup_{n \in J^0_k}C(m_n,h))
    < 2\n.
\end{align*}
Therefore,
\begin{align*}
    \int_{T^*M} |I^0_N \circ F_{\e,h}-I^1_N|^r&\leq \sum_{\footnotesize\begin{array}{c}k\in\llbracket 1,N\rrbracket\\0\not\in [\xi_{k-1},\xi_k)\end{array}}\int_{T^*M}A^r|\1_{V^0_k}\circ F_{\e,h}-\1_{V^1_k}| < 2 N A^r\n,
\end{align*}
so that 
\[\|I_N^0\circ F_{\e,h}-I^1_N\|_{L^r}< A(2N\n)^{1/r}.\]
In conclusion, taking $\n = \frac{1}{2N}(\frac{\e}{A})^{r}$,
\begin{align*}
    \|\ro_0 \circ F_{\e,h}-\ro_1\|_{L^r} &\leq \|\ro_0 \circ F_{\e,h}-I^0_N \circ F_{\e,h}\|_{L^r}+\|I_N^0 \circ F_{\e,h}-I_N^1\|_{L^r}+\|I_N^1 - \ro_1\|_{L^r}\\
    &= \|\ro_0 -I^0_N \|_{L^r}+\|I_N^0 \circ F_{\e,h}-I_N^1\|_{L^r}+\|I_N^1 - \ro_1\|_{L^r}< 3\e.
\end{align*}
\end{proof}
\subsubsection{Approximation by a Hamiltonian diffeomorphism}
The following lemma guarantees that the previously introduced permutation of the mesh can be approximated arbitrarily well by a Hamiltonian diffeomorphism. 
\begin{lemma} \label{lem10}
    Given $\ro \in L^r(T^*M)$ and $\e >0$, there exists $h_0 >0$ such that for every $h \in (0,h_0)$, every mesh $M_h$ of size $h$, and every permutation  $F_{h}$ of $M_h$, there exists $\phi \in \DHa(T^*M)$ such that 
    $$\|\ro\circ F_h - \ro\circ \phi\|_{L^r} < \e.$$
\end{lemma}
\begin{proof}
 \underline{First step:} Since the set of smooth functions with compact support  $\mathcal{C}^{\infty}_c(T^*M)$ is dense in $L^r(T^*M) $, there exists a smooth function $\Tilde{\ro}$ with compact support such that $\|\Tilde{\ro}-\ro\|_{L^r} < \e$. 
So in the following we assume that $\ro \in \mathcal{C}^{\infty}_c(T^*M)$ and we denote its compact support by $K \subset V$. 
\\ \\
\underline{Second step:} 
By compactness, for every $h>0$ and every mesh $M_h$, $K$ can be covered by a finite number of cubes of $M_h$, that is, there exists $N \in \N$ (depending on $h$ and $M_h$) such that $K \subset \cup_{n=0}^N\ov{C}(m_n,h)$. 
Since $\ro$ is continuous over the compact $K$, then it is uniformly continuous. So 
for every $h>0$ small enough there exists $\eta>0$ such that for every mesh $M_h$ of size $h$, 
$$\left\|\ro - \sum_{n=0}^N\ro(m_n)\1_{C(m_n,h-\eta)}\right\|_{L_r} < \e.$$
\\ \\
\underline{Third step:} 
Given $h>0$, a mesh $M_h$ of size $h$, and a permutation $F_h$ of $M_h$,  
we 
look for a Hamiltonian diffeomorphism $\phi$ such that $\1_{C(m_n,h-\eta)}\circ \phi$ approximates $\1_{C(m_n,h-\eta)}\circ F_h$ in $L^r(T^*M)$ for every $n \in \llbracket 0, N \rrbracket$.
The image of $m_n=(q^n,p^n)$ by $F_h$ is denoted by $\bar{m}_n=(\bar{q}^n,\bar{p}^n)$, so that  $F_h(C(m_n,h))=C(\bar{m}_n,h)$, see \cref{Permutation 1}. 
\begin{center}
\includegraphics[scale=0.35]{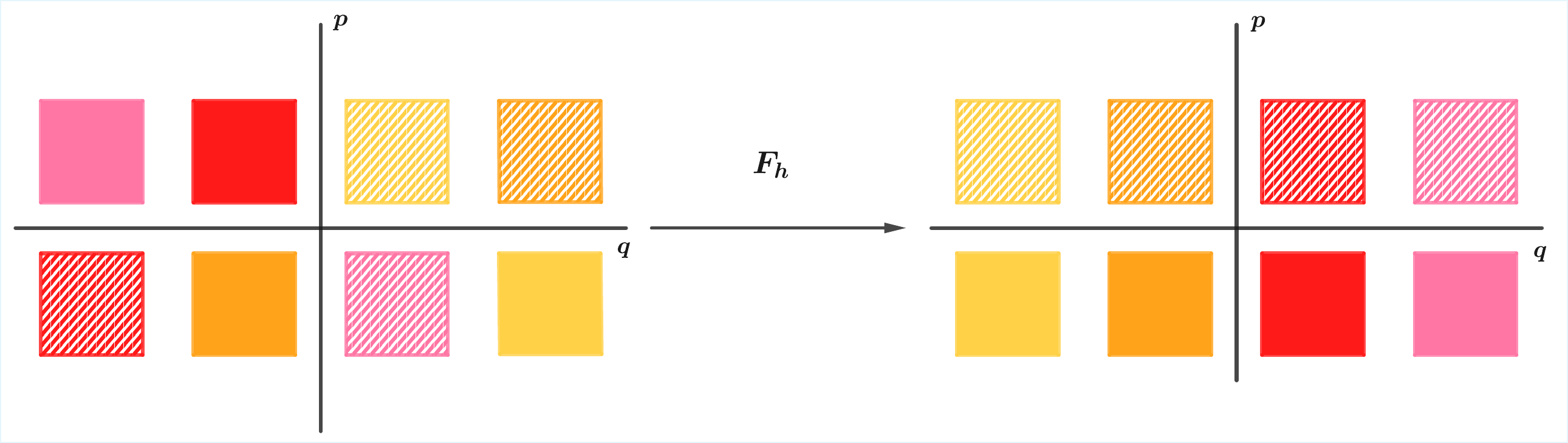} 
\captionof{figure}{The permutation $F_h$}\label{Permutation 1}
\end{center}

As a preliminary step, let us construct a Hamiltonian diffeomorphism translating each cube $C(m_n,h-\eta)$, $n\in\llbracket0,N\rrbracket$  to a cube of center $({q}^n,o^n)$, where $\|o^{n_1}-o^{n_2}\|\ge 2h$ for every $n_1\ne n_2$ in $\llbracket 0,N \rrbracket$, see \cref{Permutation 2}. 
The cubes of the mesh that cover $K$ 
are organized in columns: we choose $J\subset \llbracket0,N\rrbracket$ of minimal cardinality such that $\{q^j\mid j\in J\}=\{q^n\mid n\in \llbracket0,N\rrbracket\}$
and we set $C_j=\{n\in \llbracket0,N\rrbracket\mid q^n=q^j\}$ for $j\in J$. 
Given $\{a^j\mid j\in J\}$ in $\R^d$, we consider a function $f\in \smooth(T^*M,\R)$ with compact support and such that $f(q,p)=a^j \cdot q$ 
if $\|q-q^j\|<h-\eta$ and $p$ is at distance at most $|a^j|+h-\eta$ from $\{p^n\mid n\in C_j\}$. 
Then 
\[e^{\f{f}}(q,p)=(q,p-
a^j), \qquad \forall \: (q,p)\in C(m_n,h-\eta),\;\forall \: n\in C_j.\]
So, up to a suitable choice of the vectors $a^j$, $e^{\f{f}}$ is the required Hamiltonian diffeomorphism translating 
each cube $C(m_n,h-\eta)$, $n\in\llbracket0,N\rrbracket$,  to a cube of center $({q}^n,o^n)$, where $\|o^{n_1}-o^{n_2}\|\ge 2h$ for every $n_1\ne n_2$ in $\llbracket 0,N \rrbracket$. 
\begin{center}
\includegraphics[scale=0.30]{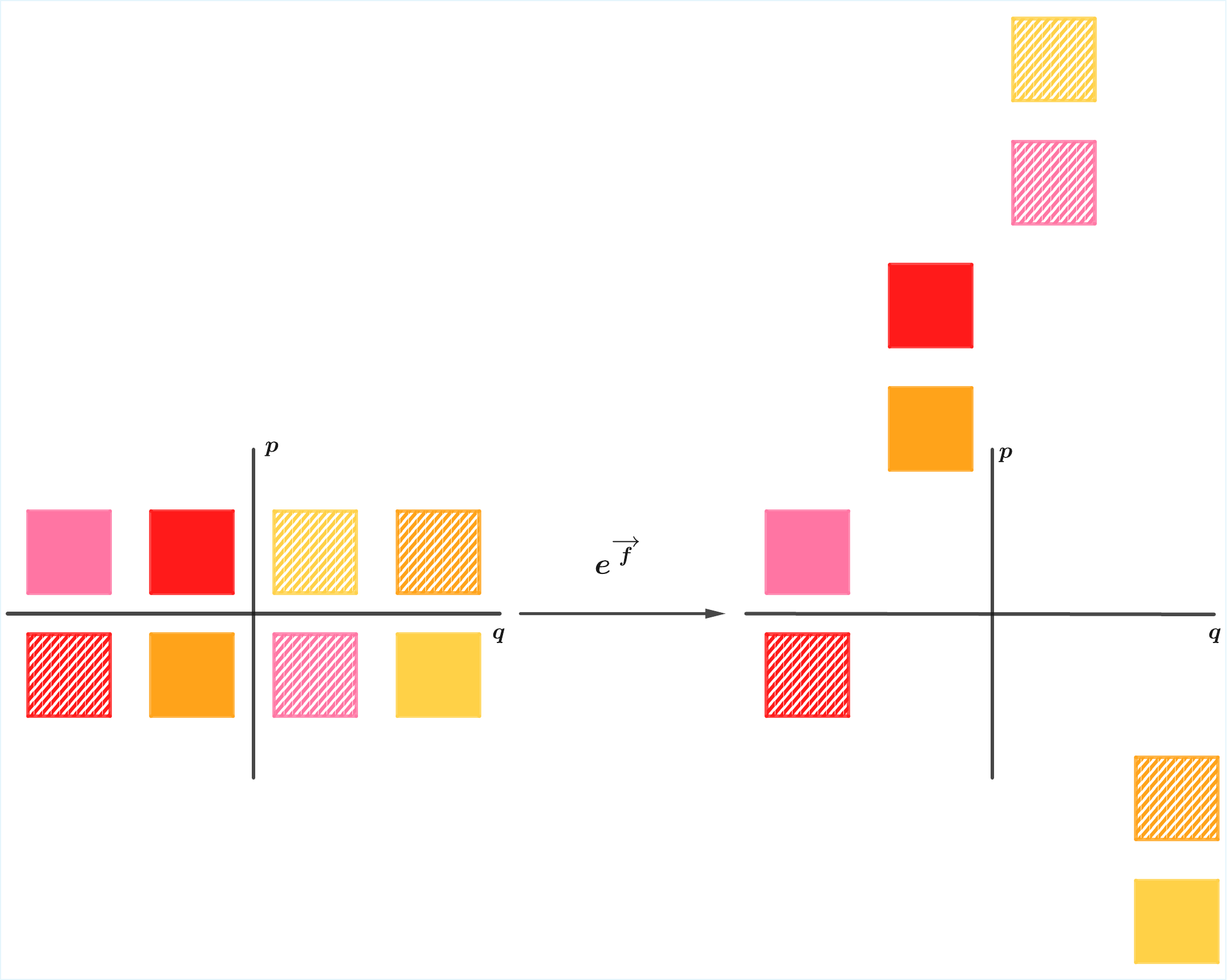}
\captionof{figure}{The Hamiltonian transformation $e^{\f{f}}$} \label{Permutation 2}
\end{center}
Then we want to translate every cube $C((q^n,o^n),h-\eta)$, $n\in\llbracket0,N\rrbracket$, to the cube $C((\bar{q}^n,o^n),h-\eta)$ by a localized horizontal translation, see \cref{Permutation 3}. 
We consider a function $g\in \smooth(T^*M,\R)$ 
with compact support and 
such that $g(q,p)=(\bar{q}^n-q^n)\cdot p$ for every 
$(q,p)$ such that $\|p-o^n\|< h-\eta$ and $q$ is at distance at most $h-\eta$ from the segment connecting $q^n$ and $\bar{q}^n$.

Then 
\[e^{\f{g}}(q,p)=(q+\bar{q}^n - q^n,p), \qquad \forall \: 
(q,p) \in C((q^n,o^n),h-\eta).\]
\begin{center}
\includegraphics[scale=0.30]{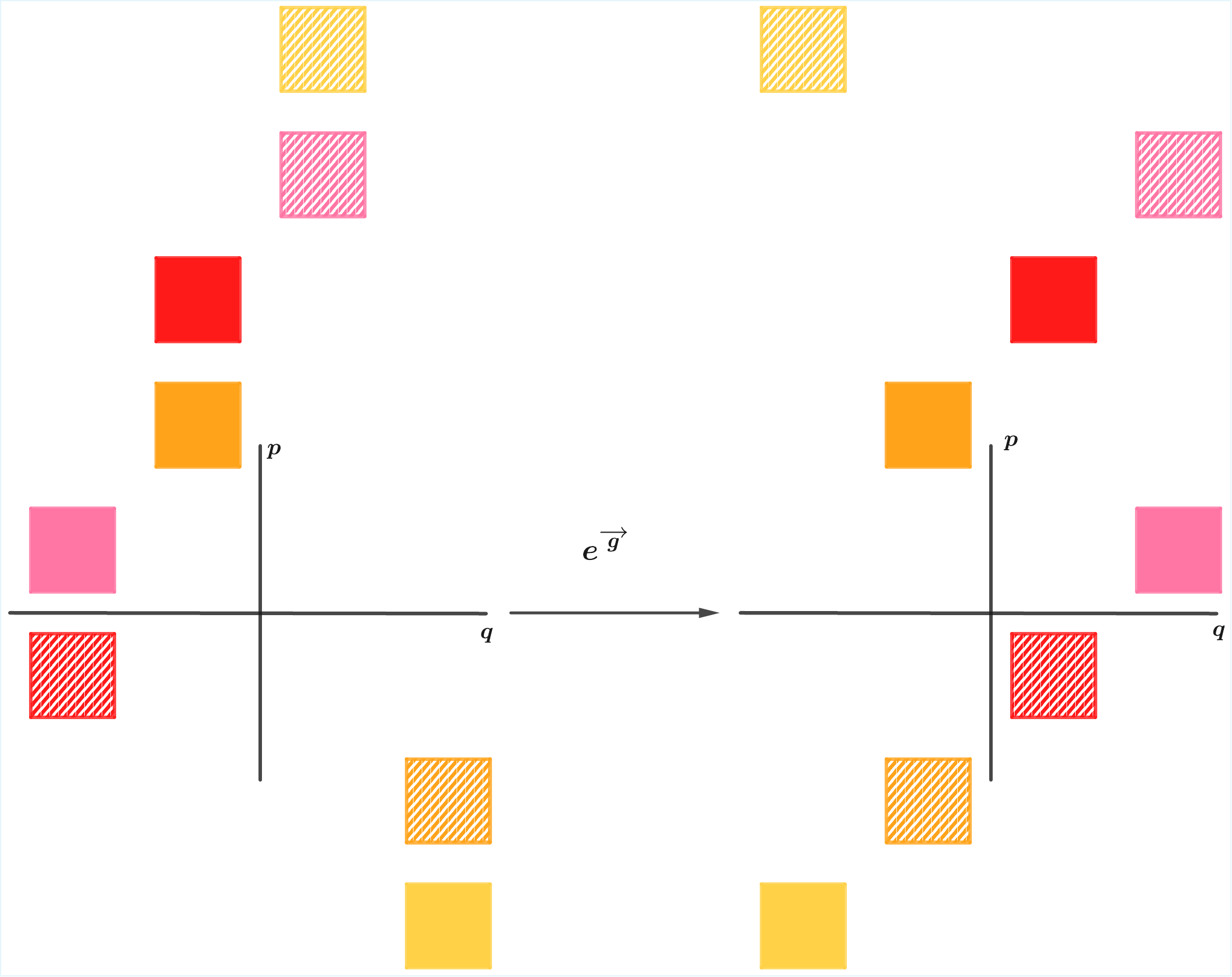}
\captionof{figure}{The Hamiltonian transformation $e^{\f{g}}$} \label{Permutation 3}
\end{center}
With the two previous transformations we obtained 
a new grid $\{(\bar{q}^n,o^n)\mid n\in\llbracket0,N\rrbracket\}$ with the good $q$-components. We look for a  Hamiltonian diffeomorphism that re-shuffles the cubes whose center 
have the same $q$-component $\bar{q}^n$ without altering the others. 
We have already seen how to approximate localized  
translations in the $p$-coordinate, see \cref{Permutation 4}. In the case where $d\ge 2$ the desired re-shuffle can be obtained by iterating sufficiently many such transitions. 
In the case $d=1$ it is sufficient to prove that there exists a Hamiltonian diffeomorphism that permutes two consecutive cubes in the same column without altering the other cubes, see \cref{Permutation 5}. 
\begin{center}
\includegraphics[scale=0.30]{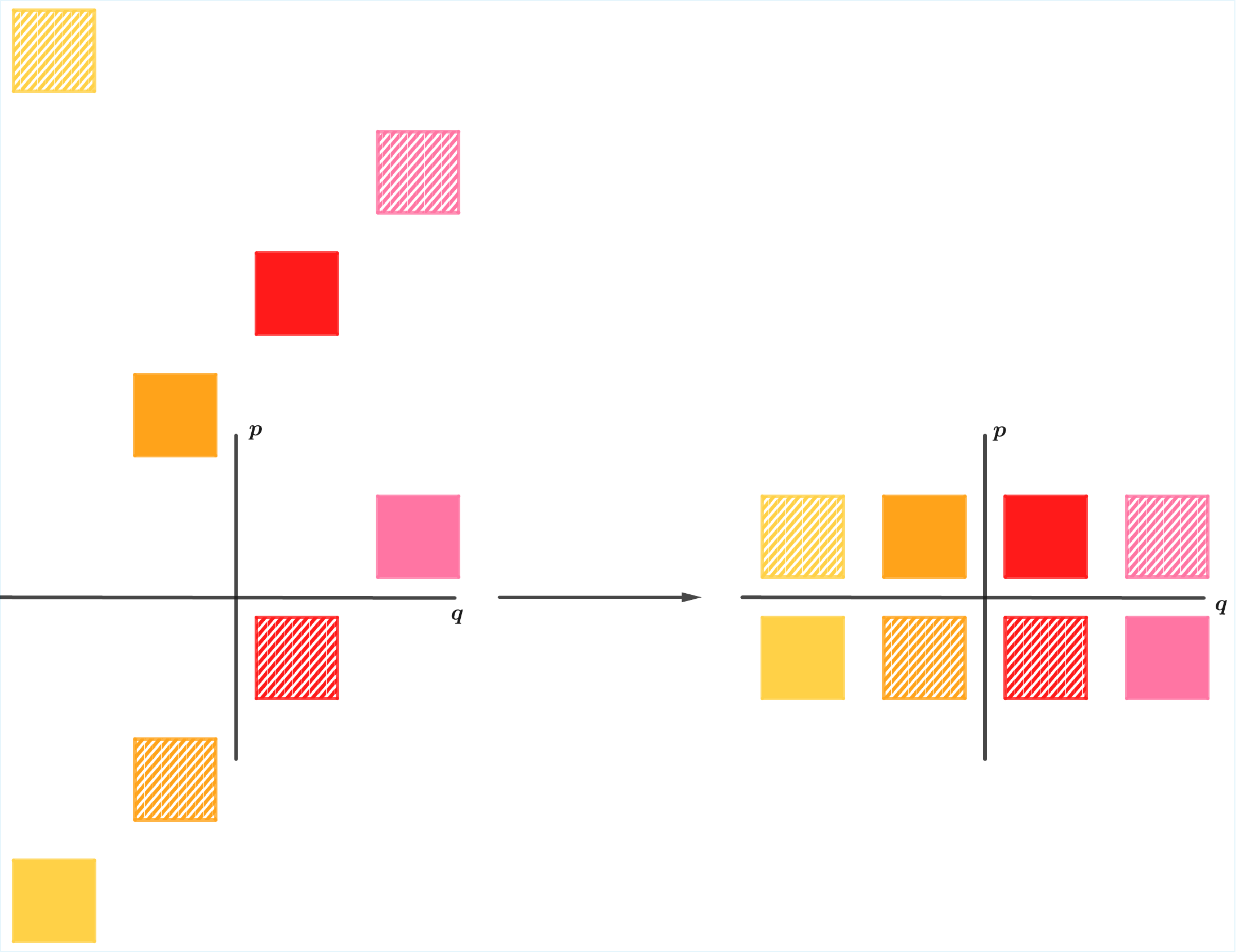}
\captionof{figure}{Localized vertical translations} \label{Permutation 4}
\end{center}
Let us consider two consecutive cubes of the same column, with centers $(\tilde{q},o^{n_1})$ and $(\tilde{q},o^{n_2})$, where $\tilde q=\bar{q}^{n_1}=\bar{q}^{n_2}$.
Consecutive means that the segment between $(\tilde{q},o^{n_1})$ and $(\tilde{q},o^{n_2})$ does not contain any other element of the grid. Moreover, we can assume that the other cubes with $q$-component equal to $\tilde{q}$ are as much separated as required.  \\
We are going to permute the cubes with a rotation of center $(\tilde{q},\tilde{p})$, where $\tilde{p} = \frac{o^{n_1}+o^{n_2}}{2}$, and of angle $\frac{\pi}{2}$. We consider a function $h_w\in \smooth(T^*M,\R)$ such that
\[h_w(q,p)=\frac{|p-\tilde{p}|^2}{2}+\frac{|q-\tilde{q}|^2}{2w^2}\]
in 
\[\Omega_w=\left\{(q,p)\mid |q-\tilde{q}|<h,\;|p-\tilde p|<\frac{h}{w}\right\},\] 
with $w>0$ to be fixed.

The Hamiltonian system associated with $h_w$ has equations 
\begin{equation} \label{eqc1} 
    \left\{ \begin{array}{ll}
         \frac{d}{dt}(q(t)-\tilde{q})=p(t)-\tilde p, \\
         \frac{d}{dt}(p(t)-\tilde p)=-\frac{1}{w^2}(q(t)-\tilde{q}),
    \end{array}
    \right.
\end{equation}
in $\Omega_w$. As long as they stay in $\Omega_w$ 
its solutions have the expression $q(t)=\tilde q+r_0 \sin(\frac tw+\theta_0)$, $p(t)=\tilde p+\frac{r_0}{w} \cos(\frac tw+\theta_0)$. Hence, for $w$ small enough, each solution of \eqref{eqc1} with initial condition in  $C((\tilde{q},o^{n_1}),h-\eta)\cup C((\tilde{q},o^{n_2}),h-\eta)$ stays in $\Omega_w$ forever. 
Moreover, we can assume that the support of $h_w$ is compact and does not intersect the cubes $C((\bar q^n,o^n),h-\eta)$ for $n\ne n_1,n_2$. 
Then $e^{w \pi\f{h_w}}$ permutes the cubes $C((\tilde{q},o^{n_1}),h-\eta)$ and $C((\tilde{q},o^{n_2}),h-\eta)$ and is the identity on $C((\bar q^n,o^n),h-\eta)$ for $n\ne n_1,n_2$.
%
\begin{center}
\includegraphics[scale=0.35]{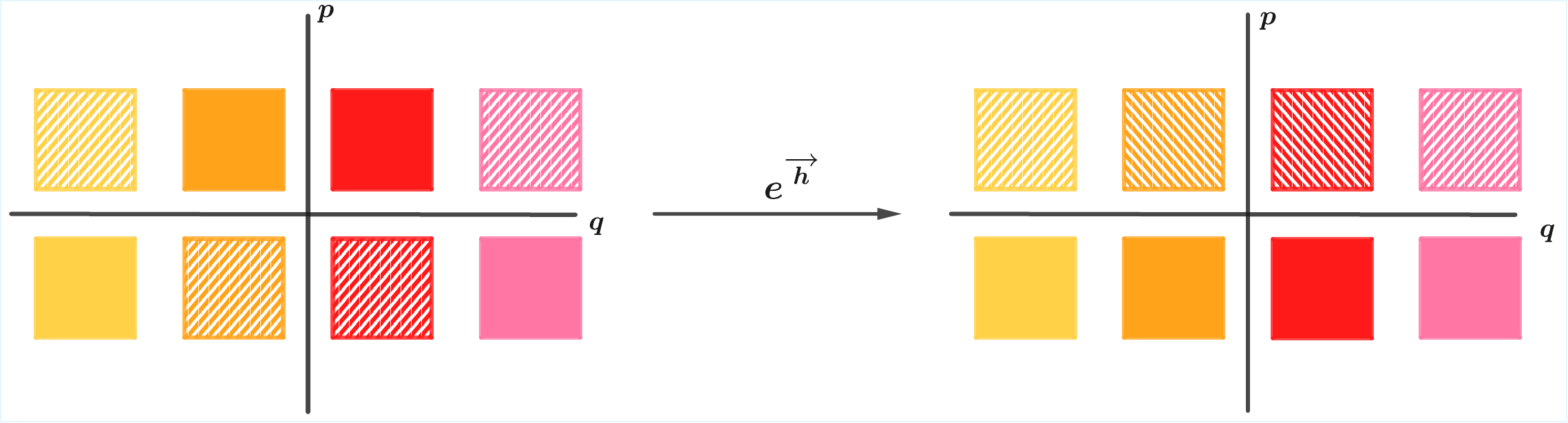}
\captionof{figure}{Localized Hamiltonian rotation of angle $\frac{\pi}{2}$} \label{Permutation 5}
\end{center}
\underline{
Fourth step:} For every $h>\eta>0$ and every permutation $F_h$ of a mesh $M_h$, we have constructed a Hamiltonian diffeomorphism $\phi$ such that $\phi(C(m_n,h-\eta))=F_h(C(m_n,h-\eta))$ for every $n\in\llbracket0,N\rrbracket$, so that
$$\sum_{n=0}^N\ro(m_n)\1_{C(m_n,h-\eta)}\circ F_h = \sum_{n=0}^N\ro(m_n)\1_{C(m_n,h-\eta)} \circ \phi.$$ 
Therefore, according to the second step, for $h$ smal enough, 
$$\|\ro\circ F_h - \ro \circ \phi\|_{L^r}< 2\e.$$
\end{proof}
\noindent \textbf{Proof of Theorem~\ref{thm:closure-orbit}.} Given $\ro_1 \in \LL(\ro_0)$, we must prove that $\ro_1 \in \ov{\mathcal{O}(\ro_0)}$. According to Lemma~\ref{lem9}, for every $\e >0$, every $h >0$ small enough, and every mesh $M_h$ of size $h$, there exists a permutation $F_{\e,h}$ of $M_h$ such that 
$$\|\ro_0 \circ F_{\e,h} - \ro_1\|_{L^r}< \e.$$
Moreover, according to Lemma~\ref{lem10}, up to further reducing $h$, there exists $\phi_h \in \DHa(T^*M)$ such that 
$$\|\ro_0 \circ F_{\e,h} - \ro_0\circ \phi_h\|_{L^r}< \e,$$
whence $\|\ro_1-\ro_0 \circ \phi_h\|_{L^r}< 2\e.$ \hfill$\square$

\section{Some properties of vector fields}\label{sec:propertiesVF}

We collect here some facts from the theory of ODEs  that we will extensively use later in the article. The next two  propositions (see, e.g., \cite[Theorem~8.7 and Lemma~8.10]{AS}) allow to deduce the convergence of flows from that of vector fields.
In what follows, ${\rm Vec}(M)$ denotes the space of $\smooth$ vector fields on the manifold $M$, which can be endowed with the \emph{compact-open topology}, that is, the topology identified by the family of semi-norms   $\|\cdot\|_{\ell,K}=\|\cdot\|_{C^\ell(K)}$, where $\ell\in \N$ and $K$ is compact in $M$. 
The group ${\rm Diff}(M)$ is also endowed with the  compact-open topology.

\begin{prop} \label{prop2}
Let $f_n\in {\rm Vec}(M)$, $n\in \N$, and $f\in {\rm Vec}(M)$ be complete and such that $f_n\to f$ 
for the compact-open topology of ${\rm Vec}(M)$.
Then, for any $t\in \R$, $e^{tf_n}\to e^{tf}$ 
for the compact-open topology  of ${\rm Diff}(M)$.
\end{prop}
The following property is the analog for vector fields of the Lie product formula.
\begin{prop} \label{th2}
 Let $f,g\in {\rm Vec}(M)$ be complete and such that $f+g$ is complete. Then,
 $$\lim_{n\to \infty}(e^{f/n}e^{g/n})^n=e^{f+g}, $$
for the compact-open topology  of ${\rm Diff}(M)$.
\end{prop}

In order to approximate some $\phi \in \DHa(T^*M)$ in the compact-open topology, we can apply a diagonal argument (based on the exhaustion of $M$ by compact sets) and reduce the problem to that of approximating $\phi$ in the $C^\ell$-topology on a given compact for a given $\ell\in \N$. This classical fact is recalled in the following lemma.
 
 \begin{lemma} \label{lem1}
    Let $\phi \in {\rm Diff}(T^*M)$ and $\mathcal{D}\subset {\rm Diff}(T^*M)$. Assume that for every compact set $K \subset T^*M$ and every $\ell\in \N$ there exists a sequence $(\varphi_n)_{n\in \N}$ in $\mathcal{D}$ such that 
    $$\|\varphi_n-\phi\|_{\ell,K}\underset{n \rightarrow  \infty}{\longrightarrow}0.$$
    Then there exists a sequence $(\phi_n)_{n\in \N}$ in $\mathcal{D}$ such that 
    $$\|\phi_n-\phi\|_{\ell,K}\underset{n \rightarrow  \infty}{\longrightarrow}0 \qquad \forall \: \ell \in \N, \: \forall K \subset T^*M \: \text{compact.}$$
\end{lemma}
Given $X\in\Ve(T^*M)$ its associated \emph{adjoint operator} is denoted by 
$$\ad_X :  \Ve(T^*M) \rightarrow \Ve(T^*M), \quad \ad_{X}Y=[X,Y].$$
Given $\phi\in\Diff(T^*M)$, we denote its \emph{pushforward} action on vector fields
as 
$$\phi_* :\Ve(T^*M)\rightarrow \Ve(T^*M), \quad (\phi_* X)(x)=D\phi(\phi^{-1}(x)) X(\phi^{-1}(x)).$$
Recall that if $\phi$ is a syplectomorphism and $h\in \smooth(T^*M,\R)$, then
\begin{equation}\label{prop:conjugation}
    \phi e^{\f{h}}\phi^{-1}=e^{\phi_* \f{h}}=e^{\f{h\circ \phi}}.
\end{equation}


The following estimate can be found in \cite[Section~8]{agrachev-sarychev2} (see also \cite{agrachev-gam}).
\begin{prop} \label{prop1} 
    Let $\ell\in \N$ and $K \subset T^*M$ be compact. There exists a compact neighborhood $K'$ of $K$ and a constant $c=c(K',\ell)>0$ such that, 
   for every $X,Y \in \Ve(T^*M)$,
    \[
    \left\|(e^{X})_*Y- \sum_{j=0}^{N-1}\frac{\ad^{j}_{X}Y}{j!}\right\|_{\ell,K}\leq ce^{c\|X\|_{\ell+1,K'}}\frac{\|X\|^N_{\ell+N,K'}}{N!}\|Y\|_{\ell+N,K'}.\]
\end{prop}

\section{Some properties of the approximately reachable sets}\label{sec:4}

\subsection{Some properties of $\ov{\mathcal{R}_{\rm st}}$}
We show in the next proposition that the small-time approximately reachable set $\ov{\mathcal{R}_{\rm st}}$ is a closed semi-group.
\begin{prop} \label{prop3}
The following holds:
\begin{itemize}
   \item[(i)] If $\phi,\psi \in \ov{\mathcal{R}_{\rm st}}$, then $\psi \circ \phi \in \ov{\mathcal{R}_{\rm st}}$.
   \item[(ii)] If $(\phi_n)_{n\in \N}\subset\ov{\mathcal{R}_{\rm st}}$ is such that $\phi_n\to \phi$, then $\phi\in \ov{\mathcal{R}_{\rm st}}$.
   \end{itemize}
\end{prop}
\begin{proof}

\noindent
\underline{Proof of (i):} Let $\ell\in \N$ and $K \subset T^*M$ be compact. Given $\e>0$ 
there exist $u(\cdot),v(\cdot)\in \PWC([0,\e],\R^m)$ and 
$\tau_u,\tau_v\in [0,\e]$
such that 
$$\|\Phi_{H_{u}}^{\tau_u}- \phi\|_{\ell,K}< \e \qquad \text{and} \qquad \|\Phi_{H_{v}}^{\tau_v}-\psi\|_{\ell,\phi(K)}< \e.$$
Consider $w(\cdot)\in \PWC([0,\tau_u+\tau_v],\R^m)$ such that
\begin{equation*}
    \left\{\begin{array}{ll}
         w(t)=u(t) \quad \text{if }t \in [0,\tau_u]  \\
         w(t)=v(t-\tau_u) \quad \text{if }t \in (\tau_u,\tau_u+\tau_v].
    \end{array} \right.
\end{equation*}
Then $\Phi_{H_{w}}^{\tau_u+\tau_u}=\Phi_{H_{v}}^{\tau_v}\circ \Phi_{H_{u}}^{\tau_u}$ and 
\begin{align*}
    \|\Phi_{H_{w}}^{\tau_u+\tau_v}-\psi \circ \phi\|_{\ell,K}&\leq \|\Phi_{H_{v}}^{\tau_v}\circ \Phi_{H_{u}}^{\tau_u}-\Phi_{H_{v}}^{\tau_v}\circ\phi\|_{\ell,K}+\|\Phi_{H_{v}}^{\tau_v}\circ \phi-\psi \circ \phi\|_{\ell,K} \\
    &= \|\Phi_{H_{v}}^{\tau_v}\circ \Phi_{H_{u}}^{\tau_u}-\Phi_{H_{v}}^{\tau_v}\circ\phi\|_{\ell,K} + \|\Phi_{H_{v}}^{\tau_v}-\psi\|_{\ell,\phi(K)} \\
    &< \|\Phi_{H_{v}}^{\tau_v}\circ \Phi_{H_{u}}^{\tau_u}-\Phi_{H_{v}}^{\tau_v}\circ\phi\|_{\ell,K} + \e.
\end{align*}
There exists a compact $\tilde{K} \subset V$ such that $\Phi_{H_{u}}^{\tau_u}(K)\cup \phi(K) \subset \tilde{K}$ independently of $\e$. 
Moreover, we can assume that there exists 
 $C>0$ depending on $\tilde K$, $\ell$, and $\phi$ (and independent of $\e$) such that
$\|\Phi_{H_{v}}^{\tau_v}\|_{\ell+1,\tilde{K}}<C$. 
By applying the mean value theorem, we obtain that 
\begin{align*}
    \|\Phi_{H_{v}}^{\tau_v}\circ \Phi_{H_{u}}^{\tau_u}
    -\Phi_{H_{v}}^{\tau_v}\circ \phi\|_{\ell,K}&\leq \|\Phi_{H_{v}}^{\tau_v}\|_{\ell+1,\tilde{K}}\|\Phi_{H_{u}}^{\tau_u}-\phi\|_{\ell,K} 
    < {C}\e.
\end{align*}
\noindent
\underline{Proof of (ii):} Let $\varepsilon>0$, $\ell\in \N$, and $K\subset T^*M$ be compact. Let $n$ be such that $\|\phi_n-\phi\|_{\ell,K}<\e$. Since $\phi_n\in\ov{\mathcal{R}_{\rm st}}$, there exist $\tau\in [0,\e]$ and   $u(\cdot)\in \PWC([0,\tau],\R^m)$ such that $\|\Phi_{{H_u}}^\tau-\phi_n\|_{\ell,K}< \e$. Hence,
$$\|\Phi_{{H_u}}^\tau-\phi\|_{\ell,K}\leq \|\Phi_{{H_u}}^\tau-\phi_n\|_{\ell,K}+\|\phi_n-\phi\|_{\ell,K}< 2\e. $$
\end{proof}
\begin{definition}\noindent
\begin{itemize}
\item A smooth function $f\in \smooth(T^*M,\R)$ is said to be \emph{STAR (small-time approximately reachable)} if $e^{s \f{f}}\in \ov{\mathcal{R}_{\rm st}}$ for all $s\in \R$.
\item A smooth function $f\in \smooth(T^*M,\R)$ is said to be \emph{STAR at the level of densities} if $\rho_0\circ e^{s \f{f}}\in \ov{{R}_{\rm st}}(\rho_0)$ for every $\rho_0\in L^r(T^*M)$ and every $s\in \R$.
\end{itemize}
\end{definition}
\begin{rem}
If a function $f\in \smooth(T^*M,\R)$ is STAR then according to Lemma~\ref{lem5} it is also STAR at the level of densities.
\end{rem}
\begin{prop} \label{prop4}
The set of STAR functions is a Lie subalgebra of $\smooth(T^*M,\R)$.
\end{prop}
\begin{proof}
The fact that if $f$ is STAR and $s\in\R$ then $sf$ is STAR is obvious, and the fact that $f+g$ is STAR if $f$ and $g$ are STAR is a direct consequence of Propositions~\ref{th2} and \ref{prop3}.
 
 Let us assume that $f,g$ are STAR and prove that $\left\{f,g\right\}$ is STAR. 
 According to Proposition~\ref{prop3}, 
$$e^{\tau \f{f}}e^{\frac{1}{\tau}\f{g}}e^{-\tau \f{f}} \in \ov{\mathcal{R}_{\rm st}} \qquad \forall \: \tau >0.$$
  According to Proposition~\ref{prop1},
\begin{align*}
    e^{\tau \f{f}}e^{\frac{1}{\tau}\f{g}}e^{-\tau \f{f}} &= \exp(\frac{1}{\tau}({e^{\tau \f{f}}})_*\f{g}) 
    =\exp(\frac{1}{\tau}\f{g}+ \f{\left\{f,g\right\}}+\f{w(\tau)}),
\end{align*}
where, for every $\ell \in \N$ and $K \subset T^*M$ compact, 
$\lim_{\tau\to0}\|\f{w(\tau)}\|_{\ell,K}=0$.
Since ${g}$ is STAR, applying Proposition~\ref{th2} we have
$$e^{\f{\left\{f,g\right\}}+\f{w(\tau)}} \in \ov{\mathcal{R}_{\rm st}} \qquad \forall \: \tau >0.$$
Finally, $e^{\f{\left\{f,g\right\}}+\f{w(\tau)}}\to e^{\f{\left\{f,g\right\}}}$ as $\tau \rightarrow 0$ in the compact-open topology by Proposition~\ref{prop2}
and we conclude thanks to point (ii) of Proposition~\ref{prop3}. 
\end{proof}

\begin{prop} \label{prop6}
    For a mechanical Hamiltonian of the form \eqref{eq:mechanical-hamiltonian}, $V_1,\dots,V_m$ are STAR.
\end{prop}
\begin{proof}
Let $i\in \llbracket 1,m \rrbracket$ and $s \in \R$. Consider the constant control $u=(u_1,\dots,u_m)$ such that $u_i=\frac{s}{\tau}$ for $\tau >0$ and $u_j=0$ for $j\neq i$. Then $\Phi_{H_u}^\tau$ is reachable in time $\tau$ and, according to Proposition~\ref{prop2},
$$\Phi_{H_u}^\tau=e^{\tau\f{H_0}+s\f{V_i}}\underset{\tau \rightarrow 0}{\longrightarrow}e^{s\f{V_i}}.$$
\end{proof}

\subsection{Some properties of $\ov{R_{\rm st}}(\rho_0)$}

\begin{lemma} \label{lem3}
    Let $\phi \in \Diff(T^*M)$ be such that $|\det J_{\phi}|\equiv 1$, where $J_{\phi}$ denotes the Jacobian matrix of $\phi$. If $\ro \circ \phi \in \ov{R_{\rm st}}(\ro)$ for every $\ro \in \mathcal{C}_c^{\infty}(T^*M)$, then $\ro \circ \phi \in \ov{R_{\rm st}}(\ro)$ for every $\ro \in L^r(T^*M)$. 
\end{lemma}

\begin{proof}
 Let $\e >0$ and $\ro \in L^r(T^*M)$. There exist $\ro_0 \in \mathcal{C}^{\infty}_c(T^*M)$, $\tau\in[0,\e]$, and $u(\cdot) \in \PWC([0,\tau],\R^m)$
 such that $\|\ro-\ro_0\|_{L^r}< \e$ and 
$\|\ro_0\circ \Phi_{H_u}^\tau-\ro_0 \circ \phi\|_{L^r} < \e.$ Then 
\begin{align*}
    \|\ro \circ \Phi_{H_u}^\tau-\ro \circ \phi\|_{L^r}&\leq \|(\ro-\ro_0) \circ \Phi_{H_u}^\tau\|_{L^r} + \|\ro_0 \circ \Phi_{H_u}^\tau-\ro_0\circ \phi\|_{L^r}+\|(\ro-\ro_0)\circ \phi\|_{L^r} \\
    &= 2 \|\ro-\ro_0\|_{L^r}+\|\ro_0 \circ \Phi_{H_u}^\tau-\ro_0\circ \phi\|_{L^r} 
    < 3\e. 
\end{align*}
\end{proof}
\begin{lemma} \label{lem4}
\begin{itemize}
   \item Let $\phi,\psi \in \Diff(T^*M)$ be such that $|\det J_{\phi}|\equiv 1\equiv |\det J_{\psi}|$. If $\ro \circ \phi \in \ov{R_{\rm st}}(\ro)$ and $\ro \circ \psi \in \ov{R_{\rm st}}(\ro)$ for every $\ro \in L^r(T^*M)$, 
    then $\ro \circ \phi \circ \psi \in \ov{R_{\rm st}}(\ro)$ for every $\ro \in L^r(T^*M)$.
    \item If $(\rho_n)_n\subset \overline{R_{\rm st}}(\rho)$ and $\rho_n\to \rho_\infty$ in $L^r(T^*M)$, then $\rho_\infty\in \overline{R_{\rm st}}(\rho)$.
    \end{itemize}
\end{lemma}
The proof follows the same arguments of Proposition~\ref{prop3} and is omitted. 


We conclude this section by proving that, as already announced, the small-time approximate controllability in ${\rm DHam}(T^*M)$ of the Hamilton equation \eqref{eq:hamilton} implies the small-time approximate controllability in $L^r(T^*M,\R)$ of the Liouville equation \eqref{eq:liouville} (cf. Lemma~\ref{lem5}).
\begin{proof}[Proof of Lemma~\ref{lem5}] According to Lemma~\ref{lem3}, it is sufficient to prove the result for $\ro_0 \in \mathcal{C}^{\infty}_c(T^*M)$.
Denote the support of $\ro_0$ by $K$. For $x \in K$ and $n\in \N^*$ denote the sphere of center $\Psi^{-1}(x)$ and radius $\frac{1}{n}$ by $S(\Psi^{-1}(x),\frac{1}{n})$. The distance between $x$ and the image of the previous sphere by $\Psi$ is strictly positive, i.e. ,  $d(x,\Psi(S(\Psi^{-1}(x),{1}/{n}))) > 0$. By compactness, the minimum of the previous distance as $x$ varies in $K$ has to be positive:
$$\delta_n:=\min_{x\in K}d\left(x,\Psi\left(S\left(\Psi^{-1}(x),\frac{1}{n}\right)\right)\right) > 0.$$
Moreover $\delta_n\to 0$ as $n\to +\infty$. 
Notice that if $x\in K$ and $y\in T^*M$ satisfy $d(x,y)<\delta_n$ then $$d(\Psi^{-1}(x),\Psi^{-1}(y))<\frac{1}{n}.$$
Fix a compact neighborhood $\tilde K$ of $K$. 
Since $\Psi\in \ov{\mathcal{R}_{\rm st}}$, for every $n\in \N$ there exist $\tau_n \in [0,1/n]$ and 
$u_n(\cdot)\in \PWC([0,\tau_n],\R^m)$ such that 
$$\|\Phi_{H_{u_n}}^{\tau_n}-\Psi\|_{0,\tilde K
}< \delta_n.$$
Then, in particular,  $d((\Phi_{H_{u_n}}^{\tau_n})^{-1}(x),\Psi^{-1}(x))<1/n$  for all $x \in K$, which implies that $(\Phi_{H_{u_n}}^{\tau_n})^{-1}(K)$ is contained in $\tilde K$ for $n$ large enough. 
\begin{align*}
    \int_{T^*M}|\ro_0 \circ \Phi_{H_{u_n}}^{\tau_n}-\ro_0 \circ \Psi|^r&=\int_{\tilde{K}}|\ro_0 \circ \Phi_{H_{u_n}}^{\tau_n}-\ro_0 \circ \Psi|^r 
    \leq \|D\ro_0\|_{0,K}^r\int_{\tilde{K}}d(\Phi_{H_{u_n}}^{\tau_n}(x),\Psi(x))^rdx 
    \\&
    \leq \|D\ro_0\|_{0,K}^r\Vol(\tilde{K})
    \delta_n^r. 
\end{align*}
\end{proof}
\begin{rem}\label{rem:lem5}
In the proof of Lemma~\ref{lem5} we actually proved that, if a sequence $\phi_n$ converges to $\phi$ in ${\rm Diff}(T^*M)$ for the compact-open topology and $\rho_0\in L^r(T^*M)$, $r\in[1,\infty)$, then $\rho_0\circ \phi_n\to \rho_0\circ \phi$ in $L^r(T^*M)$.
\end{rem}
In analogy to Proposition \ref{prop4}, we have the following property.
\begin{prop} \label{prop-densities}
The set of STAR functions at the level of the densities is a Lie subalgebra of $\smooth(T^*M,\R)$.
\end{prop}

\section{Vertical and horizontal shears}\label{sec:5}
We introduce two abelian subgroups of ${\rm DHam}(T^*M)$ which shall play a key role.
\begin{definition}[Vertical and horizontal shears on $M=\R^d$ or $\T^d$]
\noindent
\begin{itemize}
\item For $f\in \smooth(\R^d_q,\R)$ or $\smooth(\T^d_q,\R)$, the Hamiltonian diffeomorphism 
$$e^{\f{f}}(q,p)=(q,p-\nabla_q f(q)),$$
is called a \emph{vertical shear}. Vertical shears form an abelian subgroup of ${\rm DHam}(T^*M)$, denoted by $\mathcal{V}$.
\item For $g\in \smooth(\R^d_p,\R)$, 
the Hamiltonian diffeomorphism 
$$e^{\f{g}}(q,p)=(q+\nabla_pg(p),p),$$
is called a \emph{horizontal shear}. Horizontal shears form an abelian subgroup of ${\rm DHam}(T^*M)$, denoted by $\mathcal{T}$.
\end{itemize}
\end{definition}
Berger and Turaev recently proved the following useful density property.
\begin{theorem}(\cite[Corollary 1.1]{berger}) \label{th1}
Let $M=\R^d$ or $\T^d$ and $\Phi \in \DHa(T^*M)$. Then, for every $\varepsilon>0$, $\ell\in \N$, and $K\subset T^*M$ compact there exist $N\in\N$ and $S_1,\dots,S_N\in \mathcal{V}\cup \mathcal{T}$ such that 
$$\|\Phi-S_1\cdots S_N\|_{\ell,K}< \varepsilon. $$
\end{theorem}
As a consequence of Proposition~\ref{prop3} and Theorem~\ref{th1}, we get the following sufficient condition for small-time approximate reachability of all Hamiltonian diffeomorphisms.
\begin{cor}\label{cor:berger}
If $\mathcal{V}\subset \ov{\mathcal{R}_{\rm st}}$ and $\mathcal{T}\subset\ov{\mathcal{R}_{\rm st}}$ then $\ov{\mathcal{R}_{\rm st}}={\rm DHam}(T^*M)$.
\end{cor}
\section{Proof of Theorem~\ref{thm:euclidean}}\label{sec:6}
\subsection{Vertical shears on $T^*\R^d$}
We prove that vertical shears are small-time approximately reachable for system \eqref{eq:hamilton}, \eqref{eq:euclidean}.
\begin{theorem} \label{th5}
System \eqref{eq:hamilton}, \eqref{eq:euclidean} satisfies
$\mathcal{V}\subset \ov{\mathcal{R}_{\rm st}}$.
\end{theorem}
\begin{proof}
\underline{First step:} Let us prove that the Hamiltonians $p_1,\dots,p_d$ are STAR. The functions $q_1,\dots,q_d$ are STAR according to Proposition~\ref{prop6}. Applying \eqref{prop:conjugation}, we have
\begin{align*}
    e^{-\frac{v}{\tau}\f{q_i}}e^{\tau\f{H_0}}e^{\frac{v}{\tau}\f{q_i}}
    =\exp(\tau\f{H_0(e^{-\frac{v}{\tau}\f{q_i}})}),
\end{align*}
and $\tau H_0(e^{-\frac{v}{\tau}\f{q_i}}(q,p))=\tau(|p|^2/2+V_0(q))+vp_i+v^2/(2\tau)$. So $\tau \f{H_0 \circ e^{-\frac{v}{\tau}\f{p_i}}}=\tau \f{H_0}+v\f{p_i}$, because the constant function has null contribution to the Hamiltonian vector field. By taking the limit as $\tau \rightarrow 0$ 
we get that $e^{v\f{p_i}}$ is small-time approximately reachable for any $v\in\R$. \\ \\
\underline{Second step:} Let us show that the functions $\partial_{q_1}f,\dots,\partial_{q_d}f$ are STAR if $f$ is STAR. Applying \eqref{prop:conjugation}, we have that 
for every $v\in \R$, $\tau\ne 0$, and $i\in\llbracket1,d\rrbracket$ 
the diffeomorphism
\begin{align*}
    e^{-\frac{v}{\tau}\f{f}}e^{\tau\f{p_i}}e^{\frac{v}{\tau}\f{f}}&
    =\exp(\tau\f{p_i}+v\f{\partial_{q_i}f})
\end{align*}
is small-time approximately reachable. By taking the limit as $\tau\to 0$ we obtain the desired property.\\ \\
\underline{Third step:} Let us show that every $f\in \smooth(\R^d_q,\R)$ is STAR. 
Notice that it is enough to consider $f$ with compact support, since the restriction of the flow $e^{s\f{f}}$ on a given compact $K\subset \R^d$ 
coincides with $e^{s\f{g}}|_K$, where $g\in \smooth(\R^d_q,\R)$ coincides with $f$ on a  compact set $\tilde K\supset K$ and has compact support. 
In particular $f$ can be taken in $H^s(\R^d)$ for every $s\geq 0$. The set of linear combinations of Hermite functions is dense in $H^s(\R^d)$ for any $s\geq 0$, and hence approximates $f$ in $\smooth$ by Sobolev embeddings. 
By Propositions~\ref{prop2} and \ref{prop3}, 
we are thus left to prove that any linear combinations of Hermite functions is STAR. We define by induction an increasing sequence of sets $(\mathcal{H}_{j})_{j\in\N}$ in $\smooth(\R^d_q,\R)$ by
$\mathcal{H}_0=\Span_{\R}\left\{q\mapsto  e^{-|q|^2/2} \right\}$ and, for every $j \in \N^*$,
$$\mathcal{H}_j := \Span_{\R} \left\{  f_0+\sum_{k=1}^d \partial_{q_k} f_k \mid  f_0,\dots,f_d \in \mathcal{H}_{j-1} \right\}.$$
 Thanks to the second step, Proposition~\ref{prop6}, and the fact that linear combinations of STAR functions is STAR (cf. Proposition~\ref{prop4}), any $f\in \mathcal{H}_{\infty} := \cup_{j\in\N} \mathcal{H}_j$ is STAR. Recall that the Hermite functions of one variable $(\psi_n)_{n\in \N}$ satisfy the recurrence relations
$$\psi_0(x)=e^{-\frac{x^2}2},\qquad \psi_0'=\sqrt{\frac12}\psi_1,\qquad 
\psi_n'=\sqrt{\frac{n}{2}}\psi_{n-1}-\sqrt{\frac{n+1}{2}}\psi_{n+1},\quad n\geq 1.$$
Since each Hermite function in $\R^d$ can be written as $q\mapsto \psi_{j_1}(q_1)\cdots \psi_{j_d}(q_d)$ with $j_1,\dots,j_d\in \N$, we conclude that $\mathcal{H}_\infty$ contains all  Hermite functions.
\end{proof}

\subsection{Quadratic Hamiltonians on $T^{*}\R^d$}

This  section contains some preliminary results about an auxiliary quadratic Hamiltonian of the form
\begin{equation} \label{eq5}
    H_u(q,p)=\frac{|p|^2}{2}+u\frac{|q|^2}{2}, \qquad u\in \R.
\end{equation}
The following proposition states that for the control system associated with such 
a Hamiltonian 
any backward propagation along the drift is small-time approximately reachable.  Such a property already appeared in \cite{agrachev-bettina-eugenio}. We recall here its proof for completeness. 
\begin{prop} \label{th3}
The function $|p|^2$ is STAR for system \eqref{eq:hamilton}, \eqref{eq5}.
\end{prop}
\begin{proof}
First recall that the function $|q|^2$ is STAR according to Proposition~\ref{prop6}. 
Since the diffeomorphism $e^{\tau \frac{\f{|p|^2}}{2}}$ is reachable in time $\tau>0$ with control $u\equiv 0$, it follows that, for every $v\in \R$, 
$$e^{\frac{v}{2\tau}\f{|q|^2}}e^{\tau \frac{\f{|p|^2}}{2}}e^{-\frac{v}{2\tau}\f{|q|^2}}$$ 
is approximately reachable in time $\tau$. 
Applying \eqref{prop:conjugation}, we have that
\begin{align*}
    e^{\frac{v}{2\tau}\f{|q|^2}}e^{\tau \frac{\f{|p|^2}}{2}}e^{-\frac{v}{2\tau}\f{|q|^2}}
    &=\exp(\tau \frac{\f{|p|^2}}{2} - v \f{p\cdot q} + \frac{v^2}{2\tau}\f{|q|^2}).
\end{align*}
The function $|q|^2$ being STAR,  Proposition~\ref{th2} implies that $e^{\tau \frac{\f{|p|^2}}{2}-v\f{p\cdot q}}$ is approximately reachable in time $\tau$. Taking the limit as $\tau \rightarrow 0$, we get that the function $p\cdot q$ is STAR.

As a consequence, for every $s>0$, the dilation
 $$D_s(q,p):=e^{\ln(s)\f{p\cdot q}}(q,p)=\left(sq,\frac{1}{s}p\right)$$ 
 is small-time approximately reachable for every $s>0$. For $v\geq 0$ and $\tau >0$, $e^{\tau v \frac{\f{|p|^2}}{2}}$ is approximately reachable in time $\tau v$. Thus, the element $$D_{\frac{1}{\sqrt{\tau}}} e^{\tau v \frac{\f{|p|^2}}{2}}D_{\sqrt{\tau}}$$
is approximately reachable in time $\tau v$. We have
\begin{equation*}
    e^{\tau v\frac{\f{|p|^2}}{2}}D_{\sqrt{\tau}}(q,p)=\left(\sqrt{\tau}q+\sqrt{\tau} v p,\frac{p}{\sqrt{\tau}}\right),
\end{equation*}
and finally
\begin{align*}
    \left(D_{\frac{1}{\sqrt{\tau}}} e^{\tau v \frac{\f{|p|^2}}{2}}D_{\sqrt{\tau}} \right)(q,p)&=(q+vp,p) 
    =e^{v\frac{\f{|p|^2}}{2}}(q,p).
\end{align*}
So for every $\tau >0$, $$D_{\frac{1}{\sqrt{\tau}}} e^{\tau v \frac{\f{|p|^2}}{2}}D_{\sqrt{\tau}}=e^{v\frac{\f{|p|^2}}{2}}.$$
Taking $\tau$ arbitrarily small, we deduce that  $e^{v\frac{\f{|p|^2}}{2}}$ is small-time approximately reachable for every constant $v \geq 0$. Applying Proposition~\ref{th2}, the diffeomorphism $e^{v(\frac{\f{|p|^2}}{2}+\frac{\f{|q|^2}}{2})}$ is small-time approximately reachable for every $v \geq 0$. Notice that the latter element is periodic, hence for every $w <0$, there exists $v \geq 0$ such that $e^{w (\frac{\f{|p|^2}}{2}+\frac{\f{|q|^2}}{2})}=e^{v(\frac{\f{|p|^2}}{2}+\frac{\f{|q|^2}}{2})}$. Thus the Hamiltonian $|p|^2+|q|^2$ is STAR. Since $|q|^2$ is also STAR, we deduce from Proposition~\ref{prop4} that $|p|^2$ is STAR.
\end{proof}

\subsection{Horizontal shears on $T^*\R^d$}
We prove that horizontal shears are small-time approximately reachable for system \eqref{eq:hamilton}, \eqref{eq:euclidean}.

\begin{theorem} \label{th6}
System \eqref{eq:hamilton}, \eqref{eq:euclidean} satisfies
$\mathcal{T}\subset \ov{\mathcal{R}_{\rm st}}$.
\end{theorem}
\begin{proof}
According to Proposition~\ref{th3}, 
$|p|^2$ is STAR.
Hence, combining Theorem~\ref{th5} and Proposition~\ref{prop4}  we have that ${\rm ad}^k_{|p|^2/2}f$ is STAR for any $f\in \smooth(\R_q,\R)$, $k\in \N$.

Let $P(p)=p_1^{m_{1}}\dots p_d^{m_{d}}$  be a monomial. 
Setting 
 $m=m_1+\dots + m_d$ and 
$$f(q)=\frac{1}{m!}q_{1}^{m_{1}}\dots q_d^{m_{d}},$$
we have that $\ad^m_{\frac{|p|^2}{2}}f=
P$, showing that $P$ is STAR.
By density of polynomials in $\smooth(\R^d_p,\R)$, the proof is concluded.
\end{proof}
By virtue of Theorems~\ref{th5}, \ref{th6} and Corollary~\ref{cor:berger}, the proof of Theorem~\ref{thm:euclidean} is concluded.
\section{Proof of Theorem~\ref{thm:torus}}\label{sec:7}
\subsection{Vertical shears on $T^*\T^d$}
As in the case of $\R^d$, 
we now prove that vertical shears are small-time approximately reachable for system \eqref{eq:hamilton}, \eqref{eq:torus}.

\begin{theorem} \label{th7}
System \eqref{eq:hamilton}, \eqref{eq:torus} satisfies
$\mathcal{V}\subset \ov{\mathcal{R}_{\rm st}}$.
\end{theorem}
\begin{proof}
We first claim that, if $f\in \smooth(\T^d_q,\R)$ is STAR, then $e^{u\f{|\nabla_qf|^2}}\in \ov{\mathcal{R}_{\rm st}}$ for every $u \geq 0$. Applying \eqref{prop:conjugation} we have that the diffeomorphism
\begin{align*}
e^{\frac{\sqrt{u}}{\sqrt{\tau}}\f{f}}e^{\tau \f{H_0}}e^{-\frac{\sqrt{u}}{\sqrt{\tau}}\f{f}}
&=\exp(\tau \f{H_0}-2\sqrt{\tau u}\f{p\cdot \nabla_q f}+u |\nabla_qf|^2)
\end{align*}
is small-time approximately reachable in time $\tau>0$. By letting $\tau\to0$, the claim is proved.

We define an increasing sequence of vector spaces $(\mathcal{H}_{j})_{j\in\N}$ by setting
$$\mathcal{H}_0=\Span_\R \left\{1,\cos(k_1 \cdot),\sin(k_1 \cdot),\dots,\cos(k_d\cdot),\sin(k_d\cdot)\mid k_1,\dots,k_d \text{ as in }\eqref{eq:frequencies}\right\} \subset \smooth(\T^d_q,\R)$$
and, by induction, 
letting
$\mathcal{H}_j$, $j \in \N^*$, be the largest vector space whose elements can be written as 
$$  \varphi_0+\sum_{k=1}^N|\nabla\varphi_k|^2,\qquad  \text{with }N\in\N,\ \varphi_0,\dots,\varphi_N\in \mathcal{H}_{j-1}. $$
Let $\mathcal{H}_{\infty} := \cup_{j\in\N} \mathcal{H}_j$. Thanks to the claim 
and Propositions~\ref{prop4}, \ref{prop6}, 
any $f\in \mathcal{H}_\infty$ is STAR. Moreover, the proof of \cite[Proposition 2.6]{duca-nersesyan} shows that $\mathcal{H}_{\infty}$ contains all trigonometric polynomials. In particular, $\mathcal{H}_{\infty}$ is dense in $\smooth(\T^d_q,\R)$, and the conclusion follows from Proposition~\ref{prop3}.
\end{proof}

\subsection{A non-Hamiltonian symmetry on densities}
To the best of our knowledge, it is an open problem whether horizontal shears on $T^*\T^d$ can be approximately reached by system \eqref{eq:hamilton}, \eqref{eq:torus} or not. We thus turn our attention to the weaker property of approximately controlling the Liouville equation \eqref{eq:liouville}, \eqref{eq:torus}. At the level of densities, the system is less rigid and we can approximately reach the following non-Hamiltonian diffeomorphism.
\begin{lemma} \label{lem6}
    Let $S$ be the symmetry defined by
    $$S(q,p)=(q,-p).$$
    Then $\ro \circ S \in \ov{R_{\rm st}}(\ro)$ for every $\ro \in L^r(T^*\T^d)$. 
\end{lemma}
\noindent Before going through the proof of Lemma~\ref{lem6}, let us explain why it is useful. According to Lemma~\ref{lem6} and reasoning as in Lemma~\ref{lem4}, for $\tau >0$, the density $\ro\circ S \circ e^{\tau\f{\frac{|p|^2}{2}}}\circ S$ is approximately reachable in time $\tau$ from $\ro$. Thanks to the relation
\begin{equation}\label{eq:backwards-density}
S \circ e^{\tau\f{\frac{|p|^2}{2}}}\circ S=e^{-\tau \f{\frac{|p|^2}{2}}},
\end{equation}
we thus get that $\ro\circ e^{-\tau \f{\frac{|p|^2}{2}}}$ is approximately reachable in time $\tau$ from $\ro$. 
Hence, 
at the level of the densities, 
system \eqref{eq:liouville}, \eqref{eq:torus} 
can be approximately made behave as the time-reversion of the drift. 
It is not clear whether this 
can be done also
at the level of Hamiltonian diffeomorphisms on $T^*\T^d$.

\begin{proof}[Proof of Lemma~\ref{lem6}] According to Lemma~\ref{lem3}, it is sufficient to prove the result for $\ro \in \mathcal{C}^{\infty}_c(T^*M)$. 
As in the proof of Theorem~\ref{thm:closure-orbit}, 
we introduce a cubic mesh 
$M_h=\left\{m_n \mid n \in \N\right\}$
of $T^*M$ of size $h$. 
Recall that, by definition,  
$T^*M=\cup_{n\in \N}\ov{C}(m_n,h)$ and $\cup_{n\in \N}C(m_n,h)$ has zero-measure complement in $T^*M$. 
Let $K \subset T^*M$ be a compact set containing the support of $\ro$. Given $\e >0$, since $\ro$ is uniformly continuous over $K$, 
for $h$ small enough there exist $N\in \N$ and $\eta\in (0,h)$  such that 
$$\left\|\ro-\sum_{n=0}^N \ro(m_n)\1_{C(m_n,h-\eta)}\right\|_{L^r} < \e.$$
For every volume-preserving change of variables $\phi:T^*M\to T^*M$ (including $\phi=S$),  we have 
$$\left\|\left(\ro-\sum_{n=N}^N \ro(m_n)\1_{C(m_n,h-\eta)}\right)\circ \phi\right\|_{L^r}
<\e.$$
So it is sufficient to find $\phi \in \DHa(T^*M)$  that is small-time approximately reachable and such that $\1_{C(m_n,h-\eta)}\circ \phi=
\1_{C(m_n,h-\eta)}\circ S$ 
for every $n \in \llbracket 0,N \rrbracket$.
The image by $S$ of a cube of center $m=(q,p)$ is a cube of center $(q,-p)$. We will 
emulate the action of 
$S$, which is the symmetry with respect to the space $\left\{p=0\right\}$, by a rotation of center $(q,0)$ and angle $\frac{\pi}{2}$ on each plane $(q_i,p_i)$. The two transformations differ pointwise, but their images of a cube of center $m$ coincide, see \cref{Symmetry 1}.
\begin{center}
\includegraphics[scale=0.35]{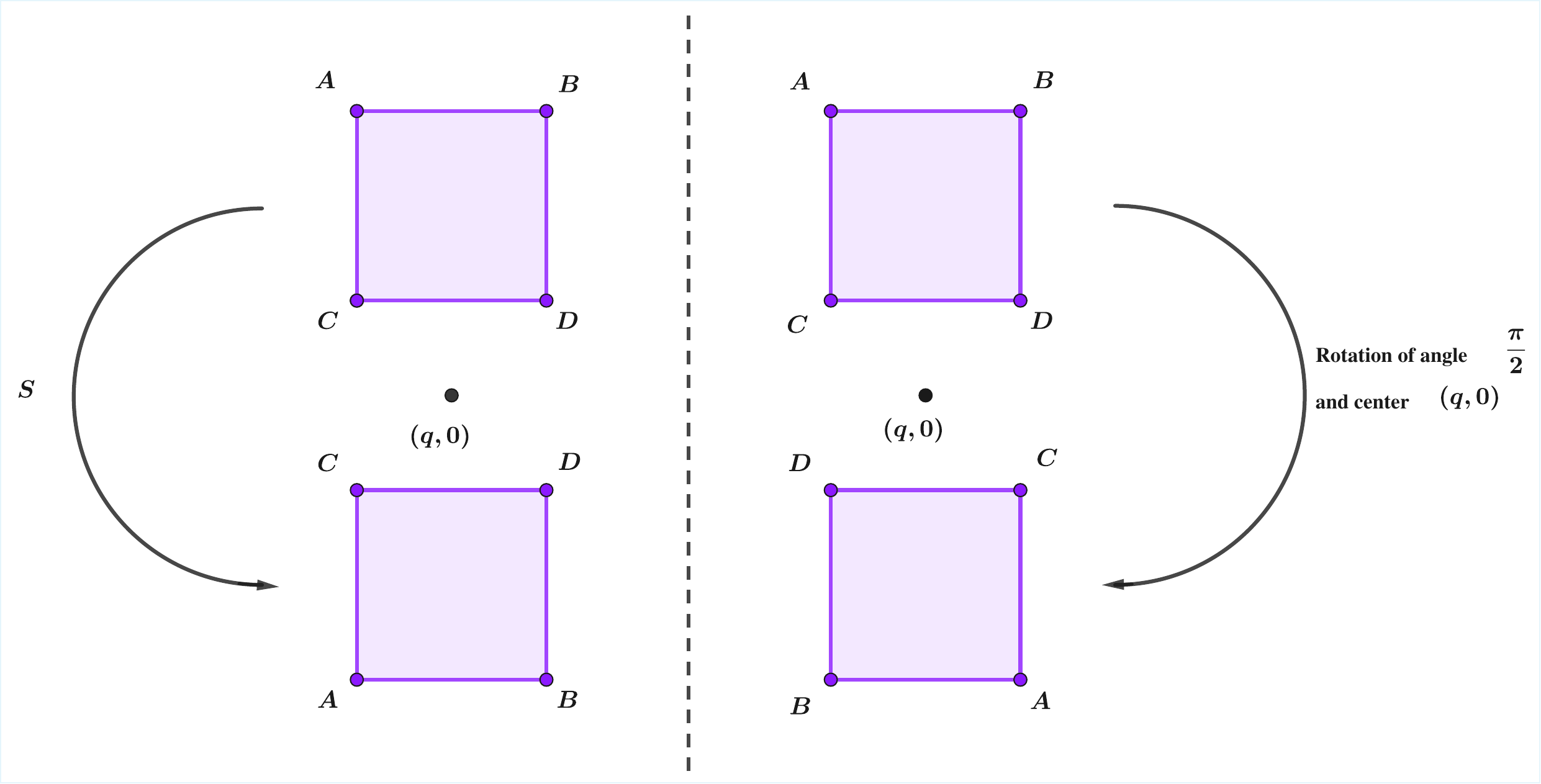} 
\captionof{figure}{Emulation of $S$ by a Hamiltonian rotation}\label{Symmetry 1}
\end{center}
As in the proof of Theorem~\ref{thm:closure-orbit}, we organize the $N+1$ cubes of radius $h-\eta$ into columns $\mathcal{C}_j$, $j\in J$, 
such that 
the centers of 
all cubes in $\mathcal{C}_j$ have the same
$q$-component $q^j$. 
We consider $j\in J$, $w
>0$ (to be fixed later) and $f \in \smooth(\T^d_q,\R)$ 
such that 
\[
f(q)=\begin{cases}
\frac{|q-q^j|^2}{2w^2} &
\mbox{if }  |q-q^j|< h,\\
0&
\mbox{if }  |q-q^j|> h+\n.
\end{cases}
\]
In particular, 
if $(q,p)$ belongs to a cube in $\mathcal{C}_{j'}$ with $j'\in J\setminus\{j\}$, then $f(q)=0$. 

According to Proposition~\ref{th2} and Theorem~\ref{th7}, $\ro \circ 
e^{T\f{(\frac{|p|^2}{2}+f)}}
$, $T>0$, is approximately reachable in time $T$. 

Integrating the Hamiltonian vector field $\frac{|p|^2}{2}+\frac{|q-q^j|^2}{2w^2}$ on 
$\Omega^j=\{(q,p)\in V\mid |q-q^j|<h\}$, 
it turns out that 
for every 
$(q,p)$ such that $|q-q^j|^2+\frac{|p|^2}{w^2}<h^2$ and  every $T>0$,
$e^{T \f{(\frac{|p|^2}{2}+f)}}(q,p)\in \Omega^j$,
with $e^{\pi w \f{(\frac{|p|^2}{2}+f})}(q^j+\tilde q,p)=(q^j-\tilde q,-p)$. 
 Choosing $w$ small enough, the ellipsoid 
\[\mathcal{E}=\left\{(q,p)\mid |q-q^j|^2+\frac{|p|^2}{w^2}<h^2\right\}\]
contains all cubes in $\mathcal{C}_j$, see \cref{Symmetry 2}. 
Set $\varphi_j=e^{\pi w \f{(\frac{|p|^2}{2}+f})}$.

\begin{center} 
\includegraphics[scale=0.35]{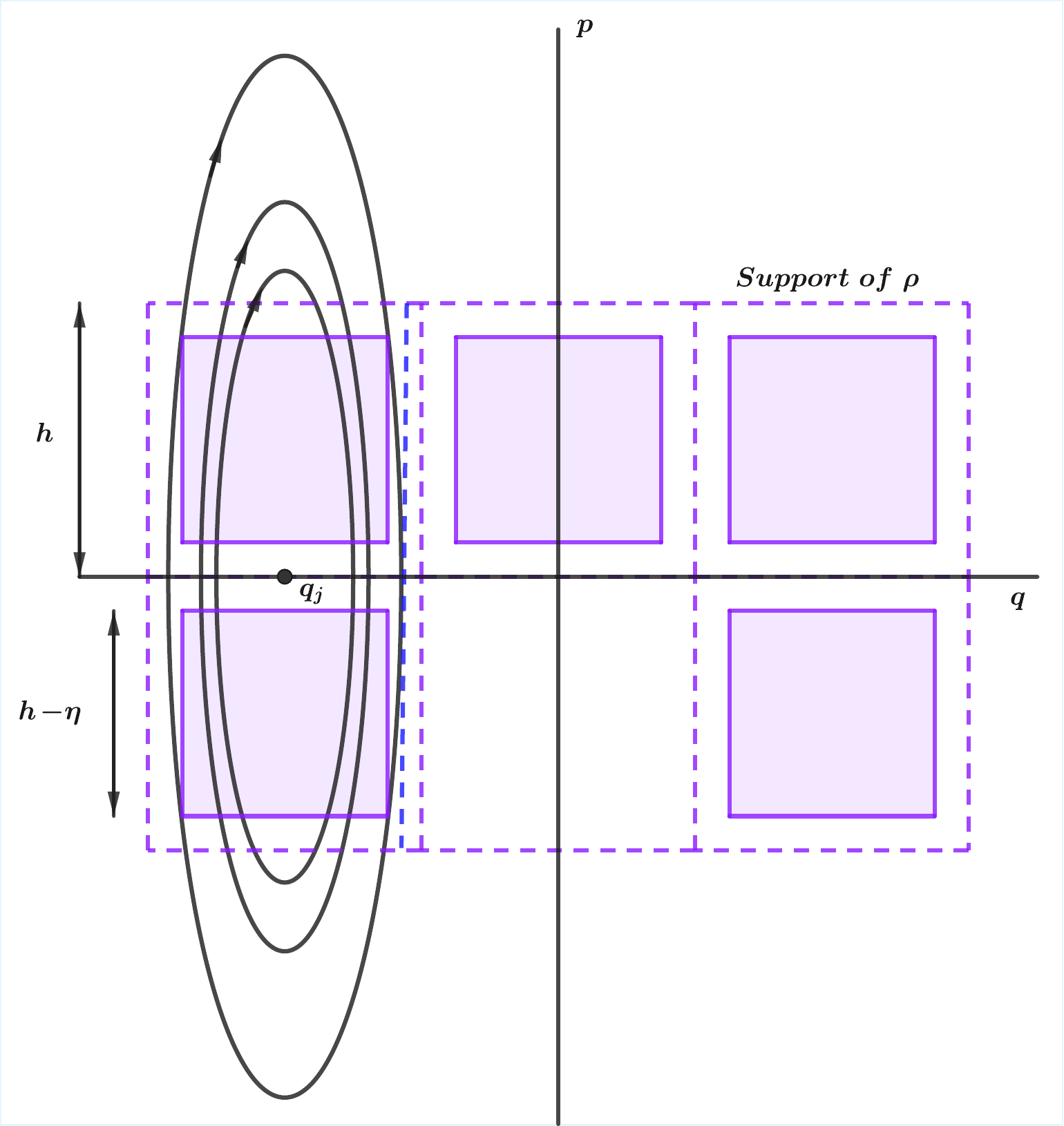}
\captionof{figure}{Permutation of two cubes in $\mathcal{C}_j$}\label{Symmetry 2}
\end{center}
After time $\pi w$, 
every cube in $\mathcal{C}_j$ is sent by $\varphi_j$ to 
its image through $S$, while $\varphi_j$ preserves all cubes in $\mathcal{C}_{j'}$ for $j'\in J\setminus\{j\}$. Moreover, $\varphi_j$ is approximately reachable in time $\pi w$, which is arbitrarily small if $w$ is. 

The proof is concluded by considering as Hamiltonian diffeomorphism $\phi$ emulating the action of $S$ on the cubes $C(m_n,h-\eta)$, $n\in \llbracket0,N\rrbracket$, the composition of all $\varphi_j$ for $j\in J$.  
\end{proof}
\subsection{Horizontal shears on $T^*\T^d$ at the level of densities}
\begin{theorem}
System \eqref{eq:liouville}, \eqref{eq:torus}, satisfies, for any $\rho_0\in L^r(T^*\T^d)$,
$$\left\{\rho_0\circ \phi\mid \phi \in\mathcal{T}\right\}\subset \ov{R_{\rm st}}(\rho_0). $$
\end{theorem}
\begin{proof}
\underline{First step.} Let us show that, if $f
\in \smooth(T^*\T^d,\R)$ is STAR at the level of densities, then 
the same is true for
$\left\{|p|^2,f\right\}$. 
According to Lemma~\ref{lem6} and relation \eqref{eq:backwards-density}, for every $\tau>0$ the density
\[\rho_0\circ e^{\tau\f{\frac{|p|^2}{2}}}e^{\frac{1}{\tau}\f{f}}e^{-\tau \f{\frac{|p|^2}{2}}}\]
is approximately reachable in time $2\tau$ from $\rho_0$. Moreover, by Proposition~\ref{prop1}, given $K \subset T^*M$ compact and $\ell\in \N$,
$$ e^{\tau\f{\frac{|p|^2}{2}}}e^{\frac{1}{\tau}\f{f}}e^{-\tau \f{\frac{|p|^2}{2}}}=\exp(\frac{1}{\tau}\f{f}+\f{\left\{\frac{|p|^2}{2},f\right\}}+\f{w(\tau)}), \qquad \mbox{with }\|\f{w(\tau)}\|_{\ell,K}=O(\tau). $$
 Since $f$ is STAR, 
 we deduce that $\rho_0\circ e^{\f{\left\{\frac{|p|^2}{2},f(q)\right\}}+\f{w(\tau)}}$ is approximately reachable in time $2\tau$ from $\rho_0$. By letting $\tau\to 0$, we obtain the desired property.
\\ \\
\underline{Second step.} 
Consider $j\in\llbracket1,d\rrbracket$ and 
$k\in \N^*$ and let us 
show that 
the function $p_j^k$ is STAR at the level of densities. 
According to Theorem~\ref{th7}, the functions $\cos(q_j)$ and $\sin(q_j)$ are STAR. Thus, by the previous step, for every $\ell\in \N$
the following functions are STAR at the level of densities:
\begin{equation*}
    \ad_{\frac{|p|^2}{2}}^{2\ell}\cos(q_j)=(-1)^\ell p_j^{2\ell}\cos(q_j),\qquad \ad_{\frac{|p|^2}{2}}^{2\ell+1}\cos(q_j)=(-1)^\ell p_j^{2\ell+1}\sin(q_j),
\end{equation*}
\begin{equation*}
    \ad_{\frac{|p|^2}{2}}^{2\ell}\sin(q_j)=(-1)^\ell p_j^{2\ell}\sin(q_j),\qquad \ad_{\frac{|p|^2}{2}}^{2\ell+1}\sin(q_j)=(-1)^{\ell+1}p_j^{2\ell+1}\cos(q_j).
\end{equation*}
In particular, writing $k=2\ell$ or $k=2\ell+1$ depending on the parity of $k$, $p_j^k\cos(q_j)$ and $p_j^k\sin(q_j)$ are STAR at the level of densities. 
Thus,
by Proposition~\ref{prop-densities},
the Hamiltonian functions
$$\frac{1}{k+1}\left\{p_j^{k+1}\cos(q_j),\sin(q_j)\right\}=p_j^k\cos^2(q_j), \qquad \frac{1}{k+1}\left\{p_j^{k+1}\sin(q_j),-\cos(q_j)\right\}=p_j^k\sin^2(q_j),$$
 are STAR at the level of the densities as well. By taking their sum, $p_j^k$ is STAR at the level of densities.
\\ \\ 
\underline{Third step.} Let us show now that every monomial $p_{1}^{m_1}\dots p_{d}^{m_d}, m_1,\dots,m_d\in \N,$ is STAR at the level of densities. We show that $p_1^mp_2^k$ is STAR at the level of densities and the method can be easily generalized to an arbitrary number of variables. 
By Proposition~\ref{prop-densities},
 the  function
$$\frac{1}{(m+2)(m+1)}\left\{ \left\{p_1^{m+2},\sin(q_1)\cos(q_2)\right\},\sin(q_1)\right\}=p_1^m\cos^2(q_1)\cos(q_2)$$ 
 is STAR at the level of densities, 
and similarly one gets that the same is true for $p_1^m\sin^2(q_1)\cos(q_2)$. Taking a linear combination of the two functions, $p_1^m\cos(q_2)$ is STAR at the level of densities. Similarly, $p_1^m\sin(q_2)$ is also STAR at the level of densities. Then,
$$\frac{1}{(k+2)(k+1)}\left\{\left\{p_2^{k+2},p_1^m\sin(q_2)\right\},\sin(q_2)\right\}=p_1^mp_2^k\cos^2(q_2),$$
is STAR at the level of densities and similarly one proves the same for $p_1^mp_2^k\sin^2(q_2)$. Finally, 
$p_1^mp_2^k$ is STAR at the level of densities, concluding the proof of the third step. 

The conclusion follows from the density of the polynomials in $\smooth(\R^d_p,\R)$.
\end{proof}

\section{Small-time exact controllability of finite ensembles of points in $T^*\R^d$ and $T^*\T^d$}\label{sec:8}

In this section, we detail that the small-time controllability of finite ensembles of points for systems \eqref{eq:euclidean} and \eqref{eq:torus} can be proved not only approximately, but also in the exact sense. This is a consequence of the fact that a finite-dimensional control systems that is approximately controllable and Lie bracket generating, is also controllable (see, e.g., \cite[Corollary~8.3]{AS}).

In what follows, $M$  denotes either $\R^d$ or $\T^d$ and $V=T^*M$. Given any $N\in\N$, let $\Delta^N \subset V^N$ be the set of $N$-uples $(\gamma_1,\dots,\gamma_N)$ with (at least) two coinciding components $\gamma_i=\gamma_j$ for some $i\neq j$. The space $V^{(N)}:=V^N \setminus \Delta^N$ has a structure of a smooth manifold. For each $\gamma \in V^{(N)}$ the tangent space $T_{\gamma}V^{(N)}$ is isomorphic to 
$T_{\gamma_1}V \times \dots \times T_{\gamma_N}V.$
First we consider a lift of the controlled systems with controlled Hamiltonian \eqref{eq:euclidean} and \eqref{eq:torus} defined respectively on $(T^*\R^d)^{(N)}$ and $(T^*\T^d)^{(N)}$. These systems are of the form \eqref{eq:hamilton}. The lift on $V^{(N)}$ is then defined by 
\begin{equation} \label{eq8bis}
    \dot{\gamma}=\f{\mathcal{H}_{u(t)}}(\gamma), \qquad \gamma=(\gamma_1,\dots,\gamma_N) \in V^{(N)},
\end{equation}
where $\mathcal{H}_u=\sum_{j=1}^{N}H_u(\gamma_j)$, the Hamiltonian $H_u$ is defined in \eqref{eq:mechanical-hamiltonian}, and
$\f{\mathcal{H}_u}(\gamma)=(\f{H_u}(\gamma_1),\dots, \f{H_u}(\gamma_N)).$ We then define the following family of Hamiltonians on $V^{(N)}$
\begin{equation*}
    \mathcal{F} = \left\{ f^{(N)}:(q^1,\dots,q^N)\mapsto \sum_{j=1}^N f(q^j) \mid f \in \smooth(M_q,\R) \right\}.
\end{equation*}
As a direct consequence of Theorems~\ref{th5} and \ref{th7}, we obtain the following Lie extension property. 
\begin{prop}
Consider the system 
\begin{equation} \label{eq9bis}
    \dot{\gamma}=\f{X}(\gamma), \qquad X\in \mathcal{F},\ \gamma \in V^{(N)}.
\end{equation}
For every $\tau >0$, the reachable set of system (\ref{eq9bis}) is contained in the small-time approximately reachable set of system (\ref{eq8bis}).
\end{prop}

Let us now prove the small-time approximate controllability of system \eqref{eq8bis}, which is an intermediate step towards the small-time exact controllability proved later in the section. 
\begin{prop}
 \label{th10}
    System (\ref{eq8bis}) is small-time approximately controllable in 
    $V^{(N)}$, i.e. , for any $N$ distinct initial configurations $(q^1,p^1),\dots,(q^N,p^N)\in V$, and $N$ distinct final configurations $(\Bar{q}^1,\Bar{p}^1),\dots,(\Bar{q}^N,\Bar{p}^N)\in V$, and $\varepsilon>0$, then there exist 
    $\tau\in[0,\varepsilon]$ and 
    $u(\cdot) \in \PWC([0,\tau],\R^m)$ such that 
    $$\|\Phi_{H_u}^\tau(q^i,p^i)-(\Bar{q}^i,\Bar{p}^i)\|< \varepsilon$$ for every $i\in\llbracket1,N\rrbracket$. 
\end{prop}

The proof of Proposition~\ref{th10} is based on the following technical result, which can be deduced for instance from Whitney extension theorem.
\begin{lemma} \label{lem2}
    For $N$ distinct positions $q^1,\dots,q^N \in M$ and vectors $a^1,\dots,a^N \in \mathbb{R}^d$, there exists a smooth function $f \in \mathcal{C}^{\infty}(M,\R)$ such that $\frac{\p}{\p q}f(q^j)=a^j$ for every $j \in \llbracket 1,N \rrbracket$.
\end{lemma}


\begin{proof}[Proof of Proposition~\ref{th10}]
 For the case of $M=\R^d$, 
 the result 
 directly follows from the small-time approximated controllability in the group $\DHa(T^*\R^d)$.
 
 Let us consider the case $M=\T^d$. 
 The pairs of initial configurations $(q^i,p^i), i\in\llbracket1,N\rrbracket,$ are distinct so, for every $\delta >0$, there exists $\tau \in [0,\delta)$ such that the points $q^1 + \tau p^1,\dots,q^N + \tau p^N$ are pairwise distinct on the torus. 
 For every $\tau >0$, the diffeomorphism $e^{\tau \f{\frac{|p|^2}{2}}}$ is approximately reachable in time $\tau$ and $e^{\tau \f{\frac{|p|^2}{2}}}(q^i,p^i)=(q^i + \tau p^i,p^i)$ for $i\in\llbracket1,N\rrbracket$. So we can assume without loss of generality that the initial positions $q^1,\dots,q^N$ are distinct.
 Similarly,  we can assume without loss of generality that the final positions $\Bar{q}^1,\dots,\Bar{q}^N$ are distinct, up to replace them by $\Bar{q}^i - \tau \Bar{p}^i$ with $\tau>0$ arbitrarily small.
 
 For each $i\in \llbracket1,N\rrbracket$, let $\hat{p}^i\in \R^d$ be such that 
$\Bar{q}^i - q^i=\hat{p}^i$ modulo $2\pi$. Applying Lemma~\ref{lem2}, there exists a function $f_{\tau} \in \smooth(\T^d,\R)$ such that $\frac{\p}{\p q}f_{\tau}(q^i)=p^i - \frac{1}{\tau}\hat{p}^i$ for every $i\in \llbracket1,N\rrbracket$. As a consequence, $e^{\f{f_{\tau}}}(q^i,p^i)=(q^i,p^i-\frac{\p}{\p q}f_{\tau}(q^i))$ and then $e^{\tau \f{\frac{|p|^2}{2}}}e^{\f{f_{\tau}}}(q^i,p^i)=(q^i+\hat{p}^i,\frac{1}{\tau} \hat{p^i})=(\Bar{q}^i,\frac{1}{\tau} \hat{p}^i)$. 
Applying again Lemma~\ref{lem2}, there exists a smooth function $g_{\tau} \in \smooth(\T^{d},\R)$ such that $\frac{\p}{\p q}g_{\tau}(\Bar{q}^i)=-\Bar{p}^i+\frac{1}{\tau}\hat{p}^i$ for every $i\in \llbracket1,N\rrbracket$. Then $e^{\f{g_{\tau}}}(\Bar{q}^i,\frac{1}{\tau}\hat{p}^i)=(\Bar{q}^i,\Bar{p}^i)$. Finally, thanks to Theorem \ref{th7}, the diffeomorphism $e^{\f{g_{\tau}}}e^{\tau \f{\frac{|p|^2}{2}}}e^{\f{f_{\tau}}}$ is approximately reachable in time $\tau$ for every $\tau > 0$. 
 \end{proof}


\begin{theorem}
    System (\ref{eq8bis}) is small-time exactly controllable 
    in $V^{(N)}$.
\end{theorem}
\begin{proof}
It is a well-known consequence of Krener's theorem (see, e.g., \cite[Corollary~8.3]{AS}\footnote{the result is not written in the \emph{small-time} case, but the proof readily extends to such a case}) that 
if a finite-dimensional control system is small-time approximately controllable and satisfies the Lie  algebra rank condition, then it is small-time 
controllable. Hence, according to Proposition~\ref{th10}, we are left to check that (\ref{eq8bis}) satisfies the Lie algebra rank condition.

 Let $\gamma=(q^1,p^1,\dots,q^N,p^N) \in V^{(N)}$. Assume for now that the positions $q^1,\dots,q^N$ are distinct in $M$. 
 Let $v^1,\dots,v^d$ be a basis of $\R^d$.
 Similarly to what is done for Lemma~\ref{lem2}, 
 by the Whitney extension theorem
 there exist  $f^1,\dots,f^d\in \smooth(M,\R)$ such that
\begin{equation}\label{eq:localization}
-\nabla_q f^i(q^1)=v^i,\quad \nabla_q f^i(q^j)=0_d,\quad {\rm Hess}_q f^i(q^1)=0_{d\times d},\quad {\rm Hess}_q f^i(q^j)=0_{d\times d},
\end{equation}
for $i \in \llbracket 1,d \rrbracket$ and $j \in \llbracket 2,N \rrbracket$.
Thanks to \eqref{eq:localization}, we have ${\f{(f^i)^{(N)}}}(\gamma)=(0_d,v_i,0_d,\dots,0_d)$.
Moreover,
using again \eqref{eq:localization}, for $i\in \llbracket 1,d \rrbracket$ we compute $$\Bigl[{\f{(f^i)^{(N)}}},\f{\mathcal{H}_0}\Bigr](\gamma)=\f{(\sum_{j=1}^N-p^j\nabla_{q^j}f^i)}(\gamma)=(v_i,0_d,\dots,0_d),$$
where $\mathcal{H}_0(\gamma)=\sum_{j=1}^NH_0(\gamma_j)$ is the lifted drift. 
Replacing $q^1$ by $q^2,\dots,q^N$ in the above argument, we can generate $2dN$ vectors that form a basis of $T_{\gamma}V^{(N)}$. 


Since
the Lie algebra $\Lie \left\{ {H_u}
\mid u \in \R^m \right\}$ is dense in $\mathcal{F}$ for the compact-open topology (cf.~Theorems~\ref{th5} and \ref{th7}), it follows by continuity that system (\ref{eq8bis}) satisfies the Lie algebra rank condition at $\gamma$.


We are left to consider the case where $q^1,\dots,q^N$ are not pairwise distinct. 
Since $\gamma \in V^N \setminus \Delta^N$, the pairs $(q^1,p^1),\dots,(q^N,p^N)$ are distinct in $V$. Then 
there exists $\delta>0$ 
such that the positions $q^1 + \delta p^1,\dots,q^N + \delta p^N$ are pairwise distinct in $M$. Then 
$\gamma'=(q^1 + \delta p^1,p^1,\dots,q^N+\delta p^N,p^N)$ 
is approximately reachable from $\gamma$ and (\ref{eq8bis}) satisfies the Lie algebra rank condition in a neighborhood of $\gamma'$.
Since each Hamiltonian $H_u$, $u\in \R^m$, is analytic, it follows that 
the Lie algebra rank condition is satisfied at every point of the orbit through $\gamma'$ for system (\ref{eq8bis}), and in particular at $\gamma$. 
\end{proof}

\textbf{Acknowledgments.}
The authors wish to thank Ivan Beschastnyi for inspiring conversations at the origin of this project, and Andrei Agrachev, Sylvain Arguill\`ere, Pierre Berger, Borjan Geshkovski, Emmanuel Tr\'elat, and Claude Viterbo for enlightening discussions.

E.P. thanks the SMAI for supporting and the CIRM for hosting the
BOUM project ”Small-time controllability of Liouville transport equations along an Hamil-
tonian field”, where some ideas of this work were conceived. 

This work has been partly supported by  the ANR-DFG project CoRoMo
ANR-22-CE92-0077-01 and the ANR project QuBiCCS
 ANR-24-CE40-3008-01.
This project has received financial support from the CNRS
through the MITI interdisciplinary programs.

\bibliographystyle{siamplain}
\bibliography{references}

\end{document}